\documentclass[twoside,a4paper,12pt,centertags]{amsart}
\usepackage{amsmath,amssymb,verbatim}
\usepackage[colorlinks,citecolor=red,pagebackref,hypertexnames=false]{hyperref}

\theoremstyle{plain}
\newtheorem{thm}{Theorem}[section]
\newtheorem{lem}[thm]{Lemma}
\newtheorem{cor}[thm]{Corollary}
\newtheorem{prop}[thm]{Proposition}

\theoremstyle{definition}
\newtheorem{defn}[thm]{Definition}
\newtheorem{rem}[thm]{Remark}

\title[BVPs for degenerate elliptic systems]
{Boundary value problems for degenerate elliptic equations and systems}
\author[P. Auscher]{Pascal Auscher}\author[A. Ros\'en]{Andreas Ros\'en}\author[D. Rule]{David Rule}
\address{Pascal Auscher, Universit\'e de Paris-Sud, UMR du CNRS 8628, 91405 Orsay Cedex, France}
\email{pascal.auscher@math.u-psud.fr}
\address{Andreas Ros\'en, Matematiska Vetenskaper, G\"oteborgs universitet, 41296 G\"oteborg, Sweden}
\email{andreas.rosen@chalmers.se}
\address{David Rule, Matematiska institutionen, Link\"opings universitet, 581 83 Link\"oping, Sweden}
\email{david.rule@liu.se}

\mathchardef\semic="303B
\newcommand{\R}{{\mathbb R}}
\newcommand{\C}{{\mathbb C}}
\newcommand{\Z}{{\mathbb Z}}
\newcommand{\mH}{{\mathcal H}}

\newcommand{\mX}{{\mathcal X}}
\newcommand{\mY}{{\mathcal Y}}

\newcommand{\mL}{{\mathcal L}}

\newcommand{\mE}{{\mathcal E}}
\newcommand{\mD}{{\mathcal D}}

\DeclareMathOperator{\re}{Re}

\newcommand{\sett}[2]{ \{ #1  \colon  #2 \} }

\newcommand{\brac}[1]{\langle #1 \rangle}
\newcommand{\supp}{\text{{\rm supp}}\,}
\newcommand{\dist}{\text{{\rm dist}}\,}

\newcommand{\nul}{\textsf{N}}
\newcommand{\ran}{\textsf{R}}
\newcommand{\dom}{\textsf{D}}

\newcommand{\clos}[1]{\overline{#1}}
\newcommand{\conj}[1]{\overline{#1}}
\newcommand{\dyadic}{\triangle}
\newcommand{\sgn}{\text{{\rm sgn}}}
\newcommand{\divv}{{\text{{\rm div}}}}
\newcommand{\curl}{{\text{{\rm curl}}}}

\newcommand{\tdd}[2]{\tfrac{\partial #1}{\partial #2}}

\newcommand{\ta}{{\scriptscriptstyle \parallel}}
\newcommand{\no}{{\scriptscriptstyle\perp}}
\newcommand{\pd}{\partial}

\newcommand{\uw}{{\underline w}}

\newcommand{\reu}{\mathbb{R}^{1+n}_+}

\newcommand{\loc}{\text{{\rm loc}}}
\newcommand{\comp}{\text{{\rm c}}}
\newcommand{\tN}{\widetilde N_*}
\newcommand{\hE}{\widehat E}

\newcommand{\tE}{\widetilde E}
\newcommand{\tS}{\widetilde S}

\newcommand{\E}{{\mathcal E}}
\newcommand{\bphi}{\varphi}
\newcommand{\bx}{{\bf x}}

\newcommand{\etaq}{1_{Q_1}}

\newcommand{\N}{{\mathbb N}}

\newcommand{\mC}{{\mathcal C}}

\newcommand{\mR}{{\mathcal R}}

\newcommand{\mS}{{\mathcal S}}

\newcommand{\eps}{\varepsilon}

\newcommand{\qe}[1]{\int_0^\infty\|#1\|^2\,\frac{dt}t}

\newcommand{\carl}[1]
{\widehat{#1}}

\def\barint_#1{\mathchoice
            {\mathop{\vrule width 6pt
height 3 pt depth -2.5pt
                    \kern -8.8pt
\intop}\nolimits_{#1}}%
            {\mathop{\vrule width 5pt height
3 pt depth -2.6pt
                    \kern -6.5pt
\intop}\nolimits_{#1}}%
            {\mathop{\vrule width 5pt height
3 pt depth -2.6pt
                    \kern -6pt
\intop}\nolimits_{#1}}%
            {\mathop{\vrule width 5pt height
3 pt depth -2.6pt
          \kern -6pt \intop}\nolimits_{#1}}}

\usepackage{color}

\definecolor{gr}{rgb}   {0.,   0.8,   0. }
\definecolor{bl}{rgb}   {0.,   0.5,   1. }
\definecolor{mg}{rgb}   {0.7,  0.,    0.7}

\makeatletter
\@namedef{subjclassname@2010}{
  \textup{2010} Mathematics Subject Classification}
\makeatother

\begin{document}

\begin{abstract} We study boundary value problems for degenerate elliptic equations and systems with square integrable boundary data. We can allow for degeneracies in the form of an $A_{2}$ weight. We obtain representations and boundary traces for solutions in appropriate classes, perturbation results for solvability and solvability in some situations. The technology of earlier works of the first two authors can be adapted to the weighted setting once the  needed  quadratic estimate is established and we even improve some results in the unweighted setting.  The proof of this quadratic estimate does not follow from earlier results on the topic and  is the core of the article.  \end{abstract}



\maketitle
\tableofcontents

\section{Introduction}

In the series of articles \cite{FKS, FJK1, FJK2}, degenerate elliptic equations in divergence form with real symmetric coefficients are studied. There, the degeneracy is given in terms of  an $A_{2}$ weight or a power of the jacobian of a quasi-conformal map. The first article gives interior estimates, the second article deals with the Wiener test and the third one study  boundary behavior and  harmonic measure.
 Further work along these lines, for example \cite{FS, BM}, has been done. However, little work on the fundamental $L_p$ Dirichlet and Neumann
problems in the degenerate setting seems to have been done.
Here, we want to initiate such study of boundary value problems, with $L_2$ boundary data, for a large class of weights, in the case of domains which are Lipschitz
diffeomorphic to the upper half space $\R^{1+n}_+ := \sett{(t,x)\in\R\times \R^n}{t>0}$, $n\ge 1$. Thus, our work includes the case of special Lipschitz domains.
Another difference to earlier work that we want to stress is that we consider general elliptic divergence form systems, and not only scalar equations, as this is the natural setting for the
methods used. In this generality, interior pointwise regularity estimates fail in general, even
in the uniform elliptic case.
However, we emphasize that the methods used in this paper do not require such pointwise estimates.

We consider
divergence form second order, real and complex, \textbf{degenerate} elliptic systems
\begin{equation}  \label{eq:divform}
   \sum_{i,j=0}^n\sum_{\beta= 1}^m \pd_i\Big( A_{i,j}^{\alpha, \beta}(t,x) \pd_j u^{\beta}(t,x)\Big) =0,\ \alpha=1,\ldots, m
\end{equation}
in $\R^{1+n}_+$,
where $\pd_0= \tdd{}{t}$ and $\pd_i= \tdd{}{x_i}$, $1\le i\le n$, which we abbreviate as
$\divv A \nabla u=0$, where
\begin{equation}   \label{eq:boundedmatrix}
  A=(A_{i,j}^{\alpha,\beta}(t, x))_{i,j=0,\ldots,n}^{\alpha,\beta= 1,\ldots,m}.  \end{equation}
We assume $A$ to be degenerate   in the sense that for some $w\in A_{2}(\R^n)$ and $C<\infty$,
  \begin{equation}   \label{eq:bounded}
  |A(t,x)| \le Cw(x),  \ \mathrm{for \ a.e.}\ (t,x)\in \R^{1+n}_{+}
  \end{equation}
and elliptic degenerate in the sense that $w^{-1}A$ is  accretive on a space $\mH^0$ that we define below. This ellipticity condition   means that
there exists $\kappa>0$ such that
\begin{equation}   \label{eq:accrassumption}
  \re \int_{\R^n}  (Af(x),f(x)) dx\ge   \kappa
   \sum_{i=0}^n\sum_{\alpha=1}^m \int_{\R^n} |f_i^\alpha(x)|^2 w(x)\, dx,
\end{equation}
for all $f\in \mH^0$ and a.e. $t>0$.  We have set $$(A\xi,\xi)= \sum_{i,j=0}^n\sum_{\alpha,\beta=1}^m A_{i,j}^{\alpha,\beta}(t,x)\xi_j^\beta\, \conj{\xi_i^\alpha}.
$$
The space $\mH^0$ is the closed subspace of $L^2(\R^n,w;\C^{m(1+n)})$ consisting of those functions with $\curl_{x}(f_{i}^\alpha)_{i=1,\ldots,n}=0$ for all $\alpha$.  The case of equations is when $m=1$ or, equivalently, when $A_{i,j}^{\alpha,\beta}=A_{i,j}\delta _{\alpha,\beta}$. In this case, the accretivity condition becomes the usual pointwise accretivity
\begin{equation}   \label{eq:poinwiseaccr}
 \re  \sum_{i,j=0}^n A_{i,j}\xi_j \conj{\xi_i} \ge
 \kappa  \sum_{i=0}^n |\xi_i|^2 w(x),
\end{equation}
for all $\xi\in\C^{1+n}$ and a.e. $(t,x)\in \R^{1+n}_{+}$.
Observe that the function $(t,x)\mapsto w(x)$ is an $A_{2}$ weight in  $\R^{1+n}$  if $w$ is an  $A_{2}$ weight in $\R^n$. So, the degeneracy is a special case of that considered in the works mentioned above. However, for the boundary value problems we wish to consider, this seems a natural class. To our knowledge,  this has not been considered before.

 A natural question is whether weights could depend on both variables or only on the $t$-variable. Already in the non-degenerate case there are regularity conditions without which the Dirichlet problem is ill-posed.  As for the degenerate case, the well known example from \cite{CaS}  when $A= t^{1-2s}I$, $0<s<1$, leading to representation of the fractional Laplacian $(-\Delta)^s$ as the Dirichlet to Neumann operator, is of  a  nature  that our theory cannot cover.  In fact, our weights need to have a trace at the boundary in some sense. Another option is to assume the weight depends on the $x$-variable only and put the $t$-perturbations in the coefficients. See below in this introduction and Section \ref{sec:rep}.  

The equation \eqref{eq:divform} must be properly interpreted.
Solutions are taken with $u$ and $\nabla u$ locally in $L^2(\R^{1+n}_{+}, w(x)dxdt)$, and the equation is taken in the sense of distributions. Note that we allow complex coefficients and systems, so most of the theory developed for real symmetric equations does not apply, and even for real coefficients we want to develop methods regardless of regularity theory.

The boundary value problems can be formulated as follows:   find weak solutions $u$ with appropriate interior estimates of $\nabla_{t,x} u$
satisfying one of the following three natural boundary conditions.
\begin{itemize}
\item The Dirichlet condition $u= \bphi$ on $\R^n$,
given the Dirichlet datum $\bphi\in L^2(\R^n,w;\C^m)$.
\item The Dirichlet regularity condition $\nabla_x u= \bphi$ on $\R^n$, given the regularity datum
$\bphi\in L^2(\R^n,w;\C^{mn})$ satisfying $\curl_x \bphi=0$. Alternately, $\bphi$ is the tangential gradient of the Dirichlet datum.
\item The Neumann condition $\pd_{\nu_{A}}u= ( A\nabla_{t,x} u, e_{0})= \bphi$ on $\R^n$, given
$\bphi\in L^2(\R^n,w^{-1};\C^{m})$. Here $e_{0}$ is the upward unit  vector in the $t$-direction.
\end{itemize}
Observe that  the natural space for  the gradient at the boundary is  $L^2(\R^n,w;\C^{m(1+n)})$, and since $A$ is of the size $w$, multiplication by $A$ maps into $L^2(w^{-1})$.  Thus  the weight $w^{-1}$ is  natural  for the conormal derivative. In order   to work with  the same weighted space for all three problems, we shall  consider the $w$-normalized conormal derivative $\partial_{\nu_{w^{-1}A}}u|_{t=0}=w^{-1}\partial_{\nu_{A}}u|_{t=0}\in L^2(\R^n, w;\C^m)$.

Boundary value problems can be formulated  for $L^p$ data with $p\ne 2$. This is for later work and here we restrict our attention to $p=2$.
 We mention that there are also ``intermediate'' boundary value problems for regularity/Neumann data in some negative Sobolev spaces based on $L^2(w)$ using  fractional powers of the Laplacian $-\Delta_{w}$ (defined later) that can be treated by the same methods.  One important case is the treatment of variational solutions in this context, \textit{i.e.} those with $\iint_{\reu} |\nabla u |^2\, dw(x) dt<\infty$, in which these problems are always well-posed. In the case $w=1$, the methods have been worked out completely in \cite{R}.

Let us also comment on the corresponding degenerate inhomogeneous problem $\divv A \nabla u=f$, with $\divv A \nabla u$ denoting the left hand side in \eqref{eq:divform}. A study of such equations would be of interest in its own right, but could also prove useful in the study of boundary value problems of the type we study here. Such applications were implemented in \cite{DinR,DPR} in the non-degenerate setting where the coefficients of the operator were assumed to satisfy a Carleson measure condition in place of $t$-independence. In \cite{DinR} a duality argument reduced the desired estimate for solutions to the homogeneous equation to an estimate on a solution to an inhomogeneous equation, which could subsequently be proved. While such investigations in the degenerate setting would certainly make the theory more complete, studying the inhomogeneous equation is not our goal here.

We shall obtain \textit{a priori} representations of solutions and existence of boundary  traces (in the almost everywhere sense) for appropriate solution spaces together with various estimates involving non-tangential maximal functions, a characterization of well-posedness, a duality principle between regularity and Dirichlet problems, perturbation results for well-posedness on perturbing the coefficients,  and well-posedness for hermitian systems and some block-triangular coefficients.

To this end, we follow the two step strategy developed in \cite{AA1} and \cite{AA2} and we assume the reader has these references handy.  The first step is to obtain a  priori representations and boundary traces (in weighted $L^2$ and almost everywhere) for solutions in appropriate classes. Actually, we shall even improve upon known results on almost everywhere convergence  when $w=1$. The starting point is to transform the second order equation to a first order system in the $w$-normalized conormal gradient
$$\nabla_{w^{-1}A} u= \begin{bmatrix} \pd_{\nu_{w^{-1}A}}u \\ \nabla_x u \end{bmatrix},$$
and then prove a weighted  quadratic estimate  and a non-tangential maximal estimate for functions of  a bisectorial operator $DB$ when the coefficients are $t$-independent.  Then one can develop a semigroup representation of those $w$-normalized  conormal gradients. This semigroup method  for $t$-independent $A$ and their $t$-dependent perturbations is presented in  such an abstract way in \cite{AA1}  that it applies in the weighted space without any change once  the initial quadratic estimate is established. We shall basically define the necessary objects and prove only what is not \textit{mutatis mutandi} the same. Here, the discrepancy function should be renormalized by the weight, so it reads $w^{-1}(x)(A(t,x) -A(0,x))$, and it is measured by a weighted Carleson-Dahlberg condition.  Also some estimates such as  weighted Carleson embedding require the fact that $dw$ is doubling, which follows from the $A_{2}$ condition.

The quadratic estimate for bisectorial $DB$ is also a fundamental ingredient in the proof of non-tangential maximal estimates and almost everywhere convergence. As in \cite{AA2}, the almost everywhere convergence is not in a pointwise sense but in some averaged sense on Whitney regions approaching the boundary.  We obtain convergence results for solutions in spaces relative to the Dirichlet problem, and for solutions,  their gradients and $w$-normalized conormal derivatives  for solutions in spaces relative to the regularity and Neumann problems. We stress that even when $w=1$, this convergence at the level of gradients is new in this generality.

 Once representations and boundary traces of solutions are obtained, the second step is to  formulate the boundary value problems and show that solvability amounts to inverting boundary operators. As for positive results, we shall obtain
well-posedness results for $t$-independent $A$ which, in addition, are hermitian, block-diagonal,  block-triangular    or  merely satisfy a divergence free condition on the triangular block. Let us remark that in the non-degenerate case $w=1$, all these BVPs are well posed with constant matrices. Here, constant would mean that  $A$ is  a constant multiple of $w$, or that  $DB$ comes with constant $B$. However, we do not know well-posedness  for  such $B$, in particular  for  the example below with $\varepsilon \sim 2$.
As usual, all positive well-posedness results are stable under perturbation, that is, when $w^{-1}(A'-A)$ is small in $L^\infty$. We also obtain perturbation results for well-posedness with $t$-dependent coefficients $A$ under  smallness of the discrepancy function $w^{-1}(A(t,x) -A(0,x))$ in a weighted Carleson-Dahlberg sense.

 The key result is therefore the quadratic estimates for bisectorial operators $DB$ in $L_2(\R^n,w)$, stated in Theorem \ref{th:main}. The differential operator $D$, although not explicit in our notation, depends on the weight $w$ as well. The structure of the operator dictates a different
reduction to a Carleson measure estimate than in the uniformly elliptic case. In particular there
appears a dyadic averaging/expectation operator $E_t$ which acts on vector valued functions by
unweighted averages of some components and weighted averages  of some other components.
It seems therefore from the proof that the $A_2$ condition on the weight is sharp for this
quadratic estimate, since the $L_2(\R^n,w)$ bound of $E_t$ requires the $A_2$ condition.
In particular, our method does not to apply to the quasiconformal Jacobians mentioned above.

The quadratic estimate from Theorem \ref{th:main} is the central estimate without which the method breaks down. We summarize the harvest  of results using in an essential way this quadratic estimate:

1. The estimates of the functional calculus in Proposition \ref{prop:equi};

2. The non-tangential maximal estimates in $L^2$ and  almost
everywhere convergence in Theorem \ref{thm:NTmaxandaeCV};

3. The estimates, through operational calculus, of the singular integral
with operator-valued kernel $S_\mE$ in Theorems \ref{thm:estSA} and \ref{thm:esttSA};

4. The perturbation results for boundary value problems in Theorems \ref{thm:Nellie} and \ref{thm: DirLip}.

The proof of the quadratic estimate is not an easy matter
and will be the hard part of the article. Let us comment on this. First, it does not follow from available extensions of the proof of the Kato square root problem such as the ones in \cite{AKMc}, \cite{elAAM} or \cite{Ban}. This is because $D$ will no longer be a constant coefficient operator.  Let us give  an example to illustrate the differences. When $w=1$, the basic example is $$
D=\begin{bmatrix}
  0     & - \frac{d}{dx}   \\
 \frac{d}{dx}     & 0
\end{bmatrix}.$$
On $\R$, if $B$ is any constant elliptic matrix, the domain of $DB$ is the Sobolev space $H^1(\R;\C^2)$. This can be seen using a Fourier transform argument. If $w\ne 1$, then $D$ becomes
$$
D=\begin{bmatrix}
  0     & - \frac{1}{w}\frac{d}{dx}w   \\
 \frac{d}{dx}     & 0
\end{bmatrix} $$
and (the column vector) $(f,g)$ is in the domain of $D$ if and only if $f\in H^1(\R,w)$ and $wg\in  H^1(\R,w^{-1})$ (the weighted Sobolev spaces).   Note that elements in these spaces are  absolutely continuous functions when $w\in A_{2}$. Let $B$ be  the constant matrix
$$
B=\begin{bmatrix}
  1     & 0   \\
\varepsilon     & 1
\end{bmatrix}.$$
Then for any $\varepsilon\ne 0$, the domain of $DB$ cannot be the same as the domain of $D$. Indeed, the conditions to be in the domain of $DB$ are $w(g+\varepsilon f) \in H^1(\R,w^{-1})$ and $f\in H^1(\R,w)$.   If $wg$ were in $H^1(\R,w^{-1})$, then so would $wf$, and  both $f$ and $wf$ would  be continuous on $\R$, which is a restriction on $f$.

Moreover, $D$ cannot be coercive on its range. In the previous example,  that would lead to an estimate like $g'$ controlled by $\frac 1 w (gw)'$, which is absurd.

 Nevertheless,  Cruz-Uribe and Rios \cite{CR} recently proved  the Kato conjecture for square roots of degenerate elliptic operators $-\divv_{\R^n} A \nabla_{\R^n}$ on $\R^n$, which  corresponds here to the special case of  block diagonal, bounded, measurable and accretive $B=\begin{bmatrix}
  1     & 0   \\
0     & w^{-1}A
\end{bmatrix}.$  The preceding example indicates that non-block terms  introduce a difficulty that was not in the above work; indeed, it makes the argument  harder. Let us explain where this difficulty appears  in the proof and how we overcome it.   As is customary in this area, we look for  a $Tb$ argument to prove a Carleson estimate which proceeds via a stopping time argument on the averages on dyadic cubes of  some test functions. In our situation, averages are taken   with respect to both the unweighted and weighted measure $dx$ and $w(x)dx$. But if the weight varies too much on a given cube $Q$,  the stopping time does
  not give  useful information.   We are forced to organize  the collection of all subcubes of $Q$ in subclasses  on which the $dx$-averages of $w$  do not vary too much from a parent cube. The ideal situation is that all cubes fall into one of these subclasses. This is not the case, but the next best thing happens, namely that  the those subcubes which are left out satisfy a packing condition with uniform, possibly large, constant. This organization of cubes is intrinsic to the given weight and is in fact a weighted version of a result found in Garnett's book \cite{Gar}. Here, it allows us to run the stopping time on the test functions for the Carleson estimate.  In the situation of \cite{CR},  one can split things so as to run two separate stopping time arguments each concerned with one measure so this step is not necessary.

To conclude this introduction,  we discuss two different boundary geometries.  First, all of the results here are invariant under bilipschitz changes of variables. Thus,  whenever one can pull back a boundary value problem on a special Lipschitz domain  to one that fits our  hypotheses, one obtains a result for the initial problem. As an example, this shows that  if $w\in A_{2}(\R^n)$,  the three boundary value problems for $-\divv_{x,t} (w(x)\nabla_{x,t}u)(x,t)=0$ with appropriate interior estimates (see Section \ref{sec:solvability}) are  well-posed on any special Lipschitz domain $\Omega=\{t>\varphi(x)\}$ when the corresponding datum is in $L^2(\pd\Omega; \tilde wd\sigma)$ where $\sigma$ is the surface measure and $\tilde w(\varphi(x)):=w(x)$. Applying the standard pullback $(x,t) \mapsto (y, t-\varphi(y))=(y,s)$, we obtain an equation of the form $ \divv_{y,s} (A(y)\nabla_{y,s}v)(y,s)=0$ that is degenerate elliptic.
Secondly, it is natural to expect results on bounded domains. Bilipschitz invariance implies that one can look at the case of the unit  ball as the non-smoothness is carried by the coefficients. There, the setup of \cite{AA2} applies to  radially independent weights and degenerate coefficients, and perturbations of the latter. It will be clear from the present article that it all depends on the weighted quadratic estimate on the boundary. This requires a proof that is left to further work.

The first  author was partially supported by the ANR project ``Harmonic analysis at its boundaries'' ANR-12-BS01-0013-01. The second one was supported by the Grant 621-2011-3744 from the Swedish Research Council, VR. The first author wants to thank Yannick Sire for bringing his attention to this problem.

\section{Preliminaries on weights}

\subsection{Muckenhoupt weights}

Recall that for a weight $w$ on $\R^n$ and $p>1$, the  $A_{p}$ condition reads
$$
\bigg(\barint_{\hspace{-6pt}Q} w\, dx\bigg)\,
\bigg(\barint_{\hspace{-6pt}Q} w^{1-p'} \, dx\bigg)^{p/p'}\le C,
$$
 for all cubes $Q$ with $p'$ the conjugate exponent to $p$.  The smallest possible  $C$ is denoted by $[w]_{A_{p}}$.
The notation $\barint_{\hspace{-2pt}E}$ means the average with respect to the indicated measure on $E$.

We identify  $w$ with the measure $dw=w(x)\, dx$ and write $w(E)$ for $\int_{E} dw$ while $|E|= \int_{E} dx$.
Recall that $w\in A_{p}$ implies $w\in A_{q}$ for all $q>p-\varepsilon$ where $\varepsilon>0$ depends on  $[w]_{A_{p}}.$  Every $w\in A_{p}$ is an $A_{\infty}$ weight:  there exist constants $0< \sigma \le 1 \le \tau<\infty$ such that
\begin{equation}
\label{eq:ainfty}
\left(\frac {|E|}{|Q|}\right)^\tau \lesssim \frac{w(E)}{w(Q)} \lesssim  \left(\frac {|E|}{|Q|}\right)^\sigma
\end{equation}
for all cubes $Q$ and measurable subsets $E$ of $Q$ (actually, $\tau=p$ if $w\in A_{p}$ and $(1/\sigma)'$ is the reverse H\"older exponent of $w$, which can be arbitrary in $(1,\infty]$). In particular, $dw$ is a doubling measure. There also exists a constant $c_0 > 0$ such that
\begin{equation} \label{reversejensen}
\barint_{\hspace{-6pt}Q} \ln w(x) dx \leq \ln\bigg(\barint_{\hspace{-6pt}Q} w(x) dx\bigg) \leq \barint_{\hspace{-6pt}Q}\ln w(x) dx + c_0,
\end{equation}
the first inequality being Jensen's inequality and the second being a reverse form of it.

Denote $ L^p(w; \C^d)=L^p(\R^n, w; \C^d)$ for $d\ge 1$, and $L^p(w)=L^p(w;\C)$.
If $w\in A_{p}$,
\begin{equation}
\label{eq:dxdwAp}
\barint_{\hspace{-6pt}Q} \left|f(y)\right|\,  dy \le [w]_{A_{p}}^{1/p}  \left( \barint_{\hspace{-6pt}Q} |f(y)|^p\,  dw(y)\right)^{1/p}.
\end{equation}
In particular, for $w\in A_{2}$, since $w\in A_{2/p}$ for some $p>1$, this implies Muckenhoupt's theorem: the Hardy-Littlewood maximal operator $M$ with respect to $dx$ is bounded on $L^2(w)$.

\subsection{A corona decomposition for $A_{2}$ weights} \label{coronasec}

We use the following dyadic decomposition of $\R^n$. Let
$\dyadic= \bigcup_{j=-\infty}^\infty\dyadic_{2^j}$ where
$\dyadic_{2^j}:=\{ 2^j(k+(0,1]^n) :k\in\Z^n \}$. For a dyadic cube $Q\in\dyadic_{2^j}$, denote by $\ell(Q)=2^j$
its \emph{sidelength} and  by $|Q|= 2^{nj}$ its {\em Lebesgue volume}.  We set $\dyadic_t=\dyadic_{2^j}$  if $2^{j-1}<t\le
2^j$.

Here we describe a decomposition of a dyadic cube $Q \in \dyadic$ with respect to a Muckenhoupt weight $w \in A_2$.

We follow a construction of Garnett \cite{Gar}. For a fixed $\sigma_w > 0$,  to be chosen later,  we consider  $B^w(Q)$ the collection of  those  (``bad for $\ln w$'')  maximal sub-cubes of $Q$ for which
\begin{equation*} \label{st1}
|(\ln w)_R - (\ln w)_Q| > \sigma_w.
\end{equation*}
 In this section, we use the notation $f_{Q}= (f)_{Q}:= \barint_{\hspace{-2pt}Q} f\, dx$.
We can then define $B^w_j(Q)$ inductively  as   $B^w_{1}(Q)=B^w(Q)$  and for $j = 2,3,\dots$ by
\[
B^w_j(Q) = \bigcup_{R \in B^w_{j-1}(Q)} B^w(R), \quad \mbox{and set} \quad B^w_*(Q) = \bigcup_{j=1}^\infty B^w_j(R).
\]
The following proposition shows that the ``number'' of  cubes {on which the oscillation of $\ln w$  differs too much from the one on some ancestor} can be controlled.

\begin{prop} \label{prop4}
For a cube $R \in \dyadic$,  if $R\subseteq Q$ is not contained in a  cube of $B^w(Q)$  then $|(\ln w)_R - (\ln w)_Q|  \leq   \sigma_w$. We also have that
\begin{equation} \label{eqsq}
\sum_{R \in B^w_*(Q)} w(R) \leq \frac{C}{\sigma_w^2} w(Q)
\end{equation}
for some $C < \infty$ which depends only on $[w]_{A_2}$.
\end{prop}

\begin{proof}
The first statement of the Proposition is clear from the definition of $B^w(Q)$, so we only need to prove \eqref{eqsq}.

For each $R \in B^w_*(Q)$, there is a unique $R'\in B^w_*(Q) \cup \{Q\}$ such that $R \in B^w_1(R')$, that is, $R'$ is the stopping-time parent of $R$. Thus
\[
w(R) \leq \int_{R} \frac{|(\ln w)_R - (\ln w)_{R'}|^2}{\sigma_w^2}\,  dw
\]
and so
\begin{align*}
\sum_{R \in B^w_*(Q)} w(R) & \lesssim \frac{1}{\sigma_w^2} \sum_{R \in B^w_*(Q)} \int_{R} |(\ln w)_R - (\ln w)_{R'}|^2\,  dw \\
& = \frac{1}{\sigma_w^2} \int_{Q} \bigg( \, \sum_{R \in B^w_*(Q)} 1_{R} |(\ln w)_R - (\ln w)_{R'}|^2 \bigg)\,  dw,
\end{align*}
where $1_R$ is the characteristic function of the set $R$.   For any $b\in L^1_{loc}(\R^n, dx)$ consider
 the square function
\[
S(b)(x) := \bigg(\, \sum_{R \in B^w_*(Q)} 1_{R}(x) |b_R - b_{R'}|^2 \bigg)^{1/2}.
\]
For any $w \in A_\infty$ we have (see  Lemma 6.4 in \cite{Gar} for the unweighted estimate and   \cite{HR2}  for the weighted one) the estimate
\[
\int_{\R^n} |S(b)|^2 dw \lesssim \int_{\R^n} |M(b)|^2 dw,
\]
where $M$ is the Hardy-Littlewood maximal operator with respect to  $dx$. Applying this to $b(x) := (\ln w(x) - (\ln w)_Q)1_Q(x)$ and using Muckenhoupt's  theorem,  it suffices to show that
\begin{equation} \label{eq:rem}
\int_Q |\ln w - (\ln w)_Q|^2\, dw \lesssim w(Q).
\end{equation}

 This inequality follows from the fact that  $\ln w \in BMO(\R^n, dx)$ when $w\in A_{2}$ and the fact that    $ BMO(\R^n, dx)= BMO(\R^n, dw)$ with equivalent norms  for $w\in A_{\infty}$. The latter statement is a consequence of the John-Nirenberg inequality and reverse H\"older inequality for $w$. However, we present a direct proof.  It follows from Jensen's inequality and the $A_2$ condition on $w$ that
\[
\int_Q e^{|\ln w(x) - (\ln w)_Q|}\, dx \lesssim |Q|.
\]
 Setting $M_\lambda = \{x \in Q \, ; \, |\ln w(x) - (\ln w)_Q| > \lambda  \}$, we can write
\[
\int_{0 }^\infty e^\lambda |M_\lambda| d\lambda=  \int_Q \int_{ 0  }^{|\ln w(x) - (\ln w)_Q|} e^{\lambda} d\lambda dx \le \int_Q e^{|\ln w(x) - (\ln w)_Q|} \, dx ,
\]
so
\[
\int_{0  }^\infty e^\lambda |M_\lambda| d\lambda \lesssim |Q|.
\]
Similarly,
\[
\int_Q |\ln w - (\ln w)_Q|^2\,  dw  = \int_{0}^\infty 2\lambda w(M_\lambda) d\lambda.
\]
Condition \eqref{eq:ainfty}, in which we take $\sigma<1$ as one can always do, ensures that
\begin{align*}
\int_{0}^\infty \lambda w(M_\lambda) d\lambda & \lesssim \int_{0}^\infty \lambda w(Q)\left(\frac{|M_\lambda|}{|Q|}\right)^\sigma d\lambda \\
& = w(Q) \int_{0}^\infty \left(\frac{|M_\lambda|}{|Q|}e^\lambda \right)^\sigma (\lambda e^{-\lambda\sigma})d\lambda \\
& \lesssim  w(Q) \left(\int_{0}^\infty \frac{|M_\lambda|}{|Q|}e^\lambda d\lambda \right)^\sigma \left(\int_{0}^\infty (\lambda e^{-\lambda\sigma})^{1/(1-\sigma)}d\lambda \right)^{1-\sigma} \\
& \lesssim w(Q),
\end{align*}
which proves \eqref{eq:rem} and with it the proof of the Proposition.
\end{proof}

\subsection{Review of weighted Littlewood-Paley inequalities}

We recall here a few facts of Littlewood-Paley theory. The treatment in Wilson's book used in \cite{CR} is not completely adapted for our needs. In particular, we use Calder\'on's reproducing formula in a different way. Here $w$ denotes a weight in $A_{2}$. We begin with approximation issues.

\begin{lem}\label{lem:approx} Let $\varphi\in L^1(\R^n)$ with a  radially decreasing integrable majorant $\phi$. Let $\varphi_{\varepsilon}(x) =\varepsilon^{-n} \varphi(x/\varepsilon)$ for $\varepsilon>0$.
\begin{itemize}
  \item Convolution with $\varphi_{\varepsilon}$ is a bounded operator on $L^2(w)$, uniformly with respect to $\varepsilon$.
  \item For every $f\in L^2(w)$, $\varphi_{\varepsilon}\star f\to  cf$ and $ \varphi_{1/\varepsilon}\star f \to 0$ in $L^2(w)$ as $\varepsilon\to 0$, where $c=\int_{\R^n}\varphi(x)\, dx$.

\end{itemize}

\end{lem}

\begin{proof} For the first point, we observe that $\sup_{\varepsilon>0}|\varphi_{\varepsilon}\star f| \le \|\phi\|_{1} Mf$ almost everywhere (See \cite[Corollary 2.1.12]{Gra}) and  recall that $M$ is bounded on  $L^2(w)$.
For the second point,
 it is easy to see that  $L^\infty_{c}(\R^n)$, the set of bounded functions with bounded support, is dense in $L^2(w)$.  Thus, it is enough to assume $f\in L^\infty_{c}(\R^n)$. Then   $\varphi_{\varepsilon}\star f- cf$ and  $\varphi_{1/\varepsilon}\star f$  converge  almost everywhere (for $dx$, thus for $dw$) to 0  as $\varepsilon\to 0$ (See \cite[Corollary 2.1.19]{Gra}). The conclusion follows using the dominated convergence theorem in $L^2(w)$ as  $Mf\in L^2(w)$.
\end{proof}

\begin{cor}\label{lem:density}  $C_{0}^\infty(\R^n)$ is dense in $L^2(w)$.  Also the space $\mE$ of
functions of the form $\int_{\eps}^R \psi_{t}\star f\, \frac{dt}t$ where $f\in C^\infty_{0}(\R^n)$, $\psi_{t}(x)=t^{-n}\psi(x/t)$  with  $\psi\in S(\R^n)$ with  Fourier transform supported away from 0 and  $\infty$,  $\varepsilon>0, R<\infty$ and $\int_{0}^\infty \hat \psi(t\xi) \, \frac {dt} t=1$ for all $\xi\ne 0$, form a dense subspace of $L^2(w)$. [Here, $\hat g$ is the Fourier transform of $g$ in $\R^n$.]
\end{cor}

\begin{proof} As  $L^\infty_{c}(\R^n)$ is dense in $L^2(w)$,
choosing $\varphi\in C^\infty_{0}(\R^n)$ in the previous lemma proves the first density. Next, with $\psi$ as in the statement, it is easy to see  that $\hat \varphi(\xi): = \int_{1}^\infty  \hat\psi(t\xi)\, \, \frac {dt} t= 1- \int_{0}^1 \hat\psi(t\xi) \, \frac {dt} t$ for $\xi\ne 0$ and $\hat \varphi(0)=1$ is a Schwartz class function and that $\int_{\eps}^{R} \psi_{t}\star f\, \frac{dt}t= \varphi_{\varepsilon}\star f - \varphi_{R}\star f$. So the lemma follows.
\end{proof}

\begin{cor}\label{cor:Calderon}  Assume $\psi$ and $\varphi$ are integrable functions, that  $\varphi$ is as in Lemma \ref{lem:approx} with $\int \varphi =1$ and   that   $\int_{\eps}^{R} \hat\psi(t\xi)\, \frac{dt}t= \hat\varphi(\varepsilon\xi)- \hat\varphi({R}\xi)$ for all $\xi\in \R^n, 0<\varepsilon<R$. Set $Q_{t}f= \psi_{t}\star f$. Then the Calder\'on representation formula $\int_{0}^\infty Q_{t}f\, \frac{dt}t= f$ holds for all $f\in L^2(w)$ in the sense that
$\int_{\varepsilon}^R Q_{t}f\, \frac{dt}t$ converges to $f$ in $L^2(w)$ as $\varepsilon\to 0$ and $R\to \infty$.
\end{cor}

\begin{proof} The lemma follows immediately from the formula
$\int_{\eps}^{R} \psi_{t}\star f\, \frac{dt}t= \varphi_{\varepsilon}\star f - \varphi_{R}\star f$ and Lemma \ref{lem:approx}.
\end{proof}

Next, we recall that any Calder\'on-Zygmund operator in the sense of  Meyer's book \cite{Me}  is bounded on $L^2(w)$ (the sharp dependance of the norm with respect to $[w]_{A_{2}}$ has been obtained by T. Hyt\"onen \cite{H} but we do not need such a refined estimate). This result extends to vector-valued Calder\'on-Zygmund operators, and consequently
we have the following useful corollary.

\begin{lem}\label{lem:LPw}  Let $\psi$ be an integrable function in $\R^n$ with $\int_{\R^n} \psi(x)\, dx=0$ and such that $(\psi_{t}(x-y))_{t>0}$ is an $ L^2({\R_+}, \frac{dt}t)$-valued Calder\'on-Zygmund kernel and set  $Q_{t}f= \psi_{t}\star f$. Then for all $f\in L^2(w)$,
$$\int_{0}^\infty \|Q_{t}f\|^2_{L^2(w)}\, \frac{dt}t \le C_{w,\psi}\|f\|^2_{L^2(w)}.$$
\end{lem}

\begin{proof}
Note that $T: f\mapsto (Q_{t}f)_{t>0}$ is a Calder\'on-Zygmund operator, bounded from $L^2(\R^n)$ to $L^2(\R^n, L^2({\R_+}, \frac{dt}t))$. Thus weighted $L^2$ theory for $A_{2}$ weights applies (see \cite{GCRF}, Chapter V).
\end{proof}

To conclude this section, we  recall a  trick proved by  Cruz-Uribe and Rios \cite{CR} originating from the proof of \cite[Corollary 4.2]{DR}.

\begin{lem}\label{lem:DRF} If $T$ is a linear operator that is bounded on $L^2(w)$ for any $w\in A_{2}$ with norm depending only on $[w]_{A_{2}}$, then for any fixed $w\in A_{2}$, there exists $\theta\in (0,1)$ and $C>0$ such that
$\|T\|_{\mL(L^2(w))}\le C\|T\|_{\mL(L^2(dx))}^\theta$.
\end{lem}

\subsection{Weighted Sobolev spaces and the gradient operator}

For $1<p<\infty$,  $W^{1,p}(w)$ is the  closure of $ C^\infty_{0}(\R^n)$ in the norm
$$
\left(\int_{\R^n} \big( |f|^p + |\nabla f|^p \big)\, dw\right)^{1/p}.
$$ For $p=2$,  we set $H^1(w)=W^{1,2}(w).$

It follows from \cite[Theorem 1.5]{FKS} that $w\in A_{p}$ satisfies a Poincar\'e inequality:
\begin{equation}
\label{eq:Poincare}
\left(\barint_{\hspace{-6pt}Q} |\psi - c |^p \, dw\right)^{1/p} \lesssim \ell(Q) \left(\barint_{\hspace{-6pt}Q} |\nabla \psi|^p \, dw\right)^{1/p}
\end{equation}
where $\psi \in C^\infty_{0}(\R^n)$, $c$ is   the mean of $\psi$ with respect to either  $dw$ or $dx$   and the implicit constant depends only on $[w]_{A_{p}}$, $n$ and $p$.  In particular, this implies that $W^{1,p}(w)$ can be identified with a space of measurable functions imbedded in $L^p(w)$.
The gradient (taken in the sense of distributions) extends to an operator on $W^{1,p}(w)$, still denoted by $\nabla$, and the Poincar\'e inequality extends to all $W^{1,p}(w)$ functions. The space of Lemma \ref{lem:density} is also dense in $W^{1,2}(w)$ when $w\in A_{2}$.

 We shall need the following result.

\begin{prop} \label{prop:interpolation} If $w\in A_{2}$, the spaces $W^{1,p}(w)$ interpolate by the real or complex methods for $r(w)< p<\infty$ where $r(w)<2$ is the infimum of  exponents $p$ for which $w\in A_{p}$.
\end{prop}

\begin{proof}
This follows from N. Badr's interpolation theorem \cite[Theorem 8.4]{Badr}, using the Poincar\'e inequality and doubling of $dw$.
\end{proof}

We denote   the Riesz transforms by $R_{j}$, $j=1,\ldots, n$.  Set   $\mR f= (R_{1}f, \ldots, R_{n}f)$ and $\mR^*g= R_{j}^*g_{1}+\cdots + R_{n}^*g_{n}$ where $R_{j}^*=-R_{j}$, the adjoint being for the $L^2(dx)$ duality. The Riesz transforms  are bounded on $L^2(w)$.

In this paper, if $A$ is an unbounded linear operator on a Banach space,  then $\dom(A), \nul(A), \ran(A)$ denote respectively, its domain, null space and range. Closure of the range is denoted by $\clos{\ran(A)}$.

\begin{lem}\label{lem:gradient} Let $w\in A_{2}$.  The following statements hold.

\begin{enumerate}
  \item  The operator $\nabla$ on $L^2(w)$ into $L^2(w;\C^n)$ with domain $\dom(\nabla)=H^1(w)$ is densely defined, closed and injective.
   \item  Its adjoint $\nabla^*:=-\divv_{w}:= - \frac 1 w \divv w$ is also closed, densely defined  and $\clos{\ran(\nabla^*)}=L^2(w)$.   Moreover,  $ \frac 1 w C^\infty_{0}(\R^n; \C^n)$ is  a dense subspace of $\dom(\divv_{w})$.
 \item The operator $\Delta_{w}=-  \nabla^*\nabla=  \divv_{w}\nabla$ constructed by the method of sesquilinear forms from $ \int_{\R^n} \nabla u \cdot \overline{\nabla v}\, dw$ on $H^1(w)$ is a  negative  self-adjoint, injective operator on $L^2(w)$.
 \item  The operator $\mR_{w}:= \nabla (-\Delta_{w})^{-1/2}$ is bounded from $L^2(w)$ into $L^2(w;\C^n)$ and its range is $\clos{\ran(\nabla)}$.
  \item   $\clos{\ran(\nabla)}= \mR(L^2(w))$.
   \item Moreover $\clos{\ran(\nabla)}= \clos{\nabla (C_{0}^\infty(\R^n))}=\{g\in L^2(w;\C^n)\, ; \, curl g=0\,\}$.

\end{enumerate}   \end{lem}

As we can see, the closure of the range of $\nabla$ has several characterizations: the one in (4) comes from spectral theory for the weighted Laplacian, the one in (5) comes from weighted theory  for the usual Laplacian.  Both will  be useful.

\begin{proof}
To prove  (1), observe that $\nabla$ is densely defined  by construction and the fact that it is closed is proved in \cite{FKS}. Injectivity follows from Poincar\'e inequalities: indeed if $u\in H^1(w)$ such that $\nabla u=0$ implies $u$ constant. As $w$ is doubling and $\R^n$ is unbounded, we have that $w(\R^n)=+\infty$, so that the constant must be 0.

The  first part of   (2) follows upon standard functional analysis from (1).  To show the density of $ \frac 1 w C^\infty_{0}(\R^n; \C^n)$ in $\dom(\divv_{w})$, we note that this is equivalent to the density of $ C^\infty_{0}(\R^n; \C^n)$ in $\dom(\divv):=\{ f\in L^2(w^{-1};\C^n)\, ; \, \divv f\in L^2(w^{-1})\}$ (equipped with the graph norm). This equivalent statement follows using regularisation by convolution as in Lemma \ref{lem:approx}  and smooth truncations.

The construction in (3) is classical (see for example, Kato's book \cite{Kato}). The injectivity follows from that of the gradient.   Next, the boundedness of $\mR_{w}$ in (4) follows from the equality  $\|\nabla f\|=\|(-\Delta_{w})^{1/2}f\|$, which itself is obtained using self-adjointness and the form. The identification of  its range follows immediately.

Let us turn to (5).
As the  Riesz transforms are bounded on $L^2(w)$ when $w\in A_{2}$ and $\mR^*\mR=I$, $\mR(L^2(w))$ is a closed subspace of $L^2(w; \C^n)$.  For $f$ in the  dense subspace $\mE$ of Corollary \ref{lem:density}, it is easy to see that $\mR f=\nabla g$ with $g=(-\Delta)^{-1/2}f \in \mS(\R^n)$ using the properties of the function $\psi$. Thus,   $\mR(\mE) \subseteq  {\ran(\nabla)}$ so that $\mR(L^2(w)) \subseteq  \clos{\ran(\nabla)}$. Conversely, if $g$ belongs to  this dense subspace  $\mE$ then $\nabla g = \mR f$ for $f=(-\Delta)^{1/2}g\in \mS(\R^n).$ Thus $\nabla (\mE) \subseteq \mR(L^2(w))$ and, by density, $\clos{\ran(\nabla)}\subseteq \mR(L^2(w))$.

 We now prove (6). We introduce $\dot H^1(w)$ as the closure of $C_{0}^\infty(\R^n)$ for the semi norm $\| \nabla f\|_{L^2(w;\C^n)}$. We claim that this space coincides with  the space of $f\in L^2_{\loc}(w)$ such that $ \nabla f\in L^2(w; \C^n)$. In one direction, if one has a Cauchy sequence $\nabla g_{k}$ in $L^2(w)$  with $g_{k}\in C_{0}^\infty(\R^n)$, then using the Poincar\'e inequality, if $c_{k}$ is the mean of $g_{k}$ on the unit ball, one can see that $g_{k}-c_{k}$ converges in $L^2_{\loc}(w)$
to some $g$ and thus $\nabla g_{k}$ converges to the distributional gradient of $g$ in $L^2(w; \C^n)$.
Conversely, if $g\in L^2_{loc}(w)$ and $\nabla g\in L^2(w; \C^n)$, then  set
$g_{k}= (g-c_{k})\varphi_{k}$  where $c_{k}$ is the mean of $g$ on the ball $B(0,2^{k+1})$ and $\varphi_{k}(x)= \varphi(2^{-k}x)$ with $\varphi$ is a smooth function that vanishes outside $B(0,2)$ and that is 1 on $B(0,1)$. It follows from Poincar\'e's inequality that $\nabla g_{k}$ converges weakly to $\nabla g$ in $L^2(w)$. By Mazur's theorem, this implies that a sequence of  convex combinations converges strongly to $\nabla g$. Noticing that $g_{k}\in H^1(w)$ in which $C_{0}^\infty(\R^n)$ is dense  concludes the proof of the claim.
Clearly, $\clos{\ran( \nabla)}= \nabla (\dot H^1(w))$, so this proves the first equality.

Next, we have $\nabla (C_{0}^\infty(\R^n)) \subseteq \{g\in L^2(w;\C^n)\, ; \, \curl g=0\,\}$. This inclusion is preserved by taking the closure in $L^2(w;\C^n)$ as the second set is clearly closed. Conversely,  using the boundedness of the Riesz transforms, the idendity
$g= (I-\mR\mR^*)g+ \mR\mR^*g$ holds for $g\in L^2(w;\C^n)$. Moreover $\divv ((I-\mR\mR^*)g)=0$ and $\curl  (\mR\mR^*g)=0$ in the distributions sense. Indeed, it is certainly true for $g\in C_{0}^\infty(\R^n;\C^n)$ and  one can use a limiting argument. If one assumes   $\curl g=0$ then one  also  has $\curl ((I-\mR\mR^*)g)=0$. Hence $(I-\mR\mR^*)g$ is constant as a distribution, and this equals to  0 since it is in $L^2(w; \C^n)$. Thus $g=\mR\mR^*g \in \mR(L^2(w))$ and we conclude using (5).
\end{proof}

\section{The main quadratic estimate}

\subsection{Review of the $DB$ formalism}

Let $\mH$ be a separable complex Hilbert space.  One uses $(\, , \, )$ and $\|\, \|$ for the hermitian product and norm on $\mH$.  Let  $D$ be an unbounded self-adjoint operator on $\mH$. There is an orthogonal  splitting
\begin{equation}\label{eq:splitting}
\mH= \nul(D) \oplus \clos{\ran(D)}.
\end{equation}

Define closed  double sectors in the complex plane by
$$
S_{\mu} := \{z\in\C : {| \arg (\pm z)|\le\mu}\}\cup\{0\}.
$$
 Consider $B\in \mL(\mH)$   and  assume that
 $B$ is (strictly) accretive on $\clos{\ran(D)}$, that is, there exists $\kappa>0$ such that
\begin{equation}
\label{accretive}
\re (Bv,v) \ge\kappa \|v\|^2, \ \forall v\in \clos{\ran(D)}.
\end{equation}
In this case, let
$$
   \mu(B):= \sup_{v\in \ran(D), v\not = 0} |\arg(Bv,v)|  <\pi/2
$$
denote the  {\em angle of accretivity} of $B$ on $\ran(D)$.
Note that $B$ may not be invertible on $\mH$. Still for $X$ a subspace of $\mH$, we set
$B^{-1}X= \{u \in \mH\, ; \,  Bu \in X\}$.

\begin{prop}   \label{prop:typeomega} With the above assumptions, we have the following facts.
\begin{itemize}
\item[{\rm (i)}]
The operator $DB$, with domain $B^{-1}\dom(D)$,  is $\mu(B)$-bisectorial, i.e. $\sigma(DB)\subseteq S_{\mu(B)}$ and there are resolvent bounds
$\|(\lambda I - DB)^{-1}\| \lesssim 1/ \dist(\lambda, S_\mu)$ when $\lambda\notin S_\mu$, $\mu(B) <\mu<\pi/2$.
\item[{\rm (ii)}]
The operator $DB$ has range $\ran(DB)=\ran(D)$ and null space $\nul(DB)=B^{-1}\nul(D)$ such that topologically (but  not necessarily orthogonally) one has
$$
\mH = \clos{\ran(DB)} \oplus \nul(DB).
$$
\item[{\rm (iii)}]
The restriction of $DB$ to $\clos{\ran(D)}$
is a closed, injective operator with dense range in
$\clos{\ran(D)}$.  Moreover, the same statements  on spectrum and resolvents as in (i) hold.

\item[{\rm (iv)}]   Statements similar to (i) and (ii) hold for $BD$ with $\dom(BD)=\dom(D)$, defined as the adjoint of $DB^*$ or equivalently by $BD= B(DB)B^{-1}$ on $\ran(BD)\cap \dom(D)$, with $\ran(BD):=B\ran(D)$ and  $BD=0$ on the null space $\nul(BD):=\nul(D)$.
\end{itemize}
\end{prop}

For a proof, see \cite{ADMc}. Note that the accretivity is only needed on $\ran(D)$.
For $t\in \R$,  we shall
\begin{align*}R_t^B&=(I+itDB)^{-1},\\
Q_t^B&= \frac{1}{2i} (R_{-t}^B -R_{t}^B)= t DB(I+t^{2}DBDB)^{-1},\\
P_{t}^B&= \frac{1}{2} (R_{-t}^B +R_{t}^B)= (I+t^{2} DBDB)^{-1}.
\end{align*}  It follows from the previous result that $R_t^B$,  $P_t^B$ and  $Q_{t}^B$
are uniformly bounded operators on $\mH$.

\subsection{The main theorem}\label{sec:main}

Let us specify in this section the operator $D$ we consider throughout.   In $\C^{m(n+1)}=\C^m\oplus \C^{mn}= \C^m \oplus (\C^m\otimes \C^n) $, we use the  notation $v=\begin{bmatrix}
        v_{\no} \\
      v_{\ta}
     \end{bmatrix}$, with $v_{\no}\in \C^m$ and $v_{\ta}\in \C^{mn}=\C^m\otimes \C^n$ which we call respectively the normal (or scalar) and tangential parts of $v$. To simplify the exposition, we carry out the detailed proof only in the ``scalar '' case $m=1$ as it carries  in the vector-valued case $m>1$ without change. See Section \ref{sec:vv} for the notation.

\begin{prop}\label{prop:D} Let $w\in A_{2}(\R^n)$ and set $\mH=L^2(w;\C^{n+1})$ with norm $\|f \|= (\int_{\R^n} |f|^2\, dw)^{1/2}$.
\begin{enumerate}
  \item The operator $D:=
     \begin{bmatrix}
        0 & \divv_{w} \\
       - \nabla  & 0
     \end{bmatrix}$ with domain  $ \begin{bmatrix}
        \dom(\nabla) \\
      \dom(\divv_{w})\end{bmatrix}$ is self-adjoint.
       \item  $\begin{bmatrix}
     C^\infty_{0}(\R^n) \\
     \frac 1 w C^\infty_{0}(\R^n; \C^n)
     \end{bmatrix}$  is a dense subspace of $\dom (D)$.
       \item  $ \clos{\ran(D)}  =  \begin{bmatrix}
       L^2(w) \\
     \mR_{w}(L^2(w))
     \end{bmatrix}.$
 \item  $ \clos{\ran(D)}  =  \begin{bmatrix}
       L^2(w) \\
     \mR(L^2(w))
     \end{bmatrix}=\{g\in L^2(\R^n,w;\C^{1+n}); \curl_x(g_\ta)=0\}$, where the closure is taken in  $\mH$  and
     $\clos{\ran(D)} \cap
      C^\infty_{0}(\R^n; \C^{n+1})$ is dense in $\clos{\ran(D)}$.
     \end{enumerate}
   \end{prop}

   The proof is a direct consequence of   Lemma \ref{lem:gradient}. We omit the  details.

\bigskip

 Let us specify the required assumption on $B$: this is the operator of multiplication by  an
$(n+1)\times(n+1)$ matrix $B(x)$ which has  bounded entries and is  accretive   on $\clos{\ran(D)}$.  Associated to $B$ are the constants $\|B\|_{\infty}$ and (the best) $\kappa>0$ in \eqref{accretive}.

 In this section,  $\|f\|$ systematically designates the  weighted  $L^2(w; \C^d)$ norm with $d=1,n$ or $n+1$ depending on the context.

\begin{thm}\label{th:main}  With the preceding assumptions, one has
\begin{equation}  \label{eq:sfBD}
  \int_0^\infty\| Q_{t}^B v \|^2 \, \frac{dt}t \lesssim \|v\|^2, \qquad  \text{for all }\ v \in \mH.
\end{equation}
 \end{thm}

 Theorem \ref{th:main} differs from previous results in the field. First, $D$ still being  a first order  differential operator,  no longer has constant coefficients. Secondly, the coercivity assumption
 $\|\nabla u\| \lesssim \|Du\|$ for $u\in \ran(D) \cap \dom(D) $ which is of utmost importance in other proofs cannot be true here, thus we cannot apply the metric measure space (here $\R^n$ with Euclidean distance and measure $dw$) generalisation made by Bandara  \cite{Ban} of the Euclidean  results  in \cite{AKMc} or \cite{elAAM}.

 Note that
 the block-matrix case  $B=\begin{bmatrix}
  1    & 0   \\
    0  &  d
\end{bmatrix}$ in the splitting $\C^{n+1}=\C \oplus \C^n$   gives a proof of  the Kato conjecture for degenerate operators, first proved by Cruz-Uribe and Rios \cite{CR}.  See Section \ref{sec:consequences}.

Our strategy to prove Theorem \ref{th:main} is to follow the line of argument in \cite{elAAM} but with some necessary twists. In particular, we will use different spectral decompositions on the scalar and  tangential parts. Also the   stopping-time argument requires the use of the Corona decomposition for $w$ (Proposition \ref{prop4}).

\subsection{Reduction to a Carleson measure estimate}\label{sec:reduction}

\begin{lem} [Off-diagonal decay] \label{lem:odd}
 For every integer $N$ there exists $C_N>0$
such that
\begin{equation} \label{odn}
\|1_{E}\,R_{t}^B u\|+ \|1_{E}\,Q_{t}^B u\| \le C_N \brac{\dist (E,F)/|t|}^{-N}\|u\|
\end{equation}
for all $t\ne 0$,
whenever $E,F \subseteq \R^n$ are closed sets,  $u \in \mH$
is such that $\supp u\subseteq F$.  We have set $\brac x:=1+|x|$ and
$\dist(E,F) :=\inf\{|x-y|\, ; \, x\in E,y\in F\}$.

\end{lem}

\begin{proof}
Repeat the argument in \cite[Proposition 5.1]{elAAM}, which only uses that $D$ is of order 1 and that the commutator $[\chi, D]$ of $D$ with a Lipschitz function $\chi$  is bounded on the  space $\mH$. More precisely,  it is the pointwise multiplication by
 \begin{equation}
\label{eq:comm}
\begin{bmatrix}
        0 &  -(\nabla \chi)^t  \\
        \nabla \chi  & 0
     \end{bmatrix}.
\end{equation}
 We omit further details.
\end{proof}

 Let  $A_t$ be the   dyadic averaging operator  with respect to Lebesgue measure and $A_{t}^w$ the one with respect  to $dw$.  Set
$$E_{t}u(x):= \begin{bmatrix}
      A_{t}^w u_{\no}(x)    \\
       A_{t}u_{\ta}(x)     \end{bmatrix} =\begin{bmatrix}
      \barint_{\hspace{-2pt}Q} u_{\no}(y)\,  dw(y)    \\
       \barint_{\hspace{-2pt}Q} u_{\ta}(y)\,  dy
     \end{bmatrix},
     $$
     where $Q$ is the   unique dyadic cube  $Q \in \dyadic_t$ that contains $x$.  The notation for dyadic cubes is the same as in Section \ref{coronasec}.  Recall that    $\barint_{\hspace{-2pt}Q}$ means the average on $Q$ with respect to the indicated measure. We also set
     $$
     E_{Q}u= \begin{bmatrix}
      \barint_{\hspace{-2pt}Q} u_{\no}(y)\,  dw(y)    \\
       \barint_{\hspace{-2pt}Q} u_{\ta}(y)\,  dy
     \end{bmatrix}.
     $$
Observe  that $E_{t}$ acts  as a linear operator componentwise.  Using the inequality
    \eqref{eq:dxdwAp} with $p=2$,
we have
   that $E_{t}$ is a bounded operator on $\mH$, uniformly in $t$.
   We also have the  pointwise estimate
   $$
\sup_{t>0}    |E_{t}u| \le   M_{d,w}|u_{\no}| + M_{d}|u_{\ta}|
   $$
   where $M_{d,w}$ is the dyadic maximal function with respect to
   $dw$ and $M_{d}$ the dyadic maximal function with respect to
   $dx$. Both are bounded on $L^2(w)$ by the Hardy-Littlewood  theorem for the doubling measure $dw$ and
   by Muckenhoupt's theorem since $w\in A_{2}$. Thus, $u\mapsto \sup_{t>0}    |E_{t}u|$ is bounded from $\mH$ into $L^2(w)$.

\begin{defn}  \label{defn:princpart}
By the {\em principal part} of $(Q_t^B)_{t>0}$,
we mean the  function   $(x,t) \mapsto \gamma_t(x)$ defined from $\R^{n+1}_+$  to $ \mL(\C^{n+1}) $ by
$$
   \gamma_t(x)z:= (Q_t^B z)(x)
$$
for every $z\in \C^{n+1}$. We view $z$ on the right-hand side
of the above equation as the constant function valued in $\C^{n+1}$ defined on $\R^n$
by $z(x):=z$.  We denote by $|\gamma_t(x)|$ its norm in $\mL(\C^{n+1})$ subordinated to the hermitian structure on $\C^{n+1}$.
We identify $\gamma_t(x)$ with the (possibly unbounded) multiplication
operator $\gamma_t: u(x)\mapsto \gamma_t(x)u(x)$, $u\in \mH$.
\end{defn}
\begin{lem}\label{lem:gammat}
The operator $Q_t^B$ extends to a bounded operator from
$L^\infty(w; \C^{n+1})$ into $ \mH_{\text{loc}}=L^2_{\text{loc}}(w; \C^{n+1})$.
In particular we have
$\gamma_t\in L^2_{\text{loc}}(w; \mL(\C^{n+1}))$ with bounds
$$
   \barint_{\hspace{-6pt}Q} |\gamma_t(y)|^2 \, dw(y) \lesssim
  1
$$
for all $Q\in\dyadic_t$.
Moreover, $\gamma_t E_{t}$ are bounded on $\mH$ with $\|\gamma_t E_{t}u\|\lesssim \|E_{t}u\|$ uniformly for all $t>0$ and $u\in \mH$.
\end{lem}

\begin{proof} Fix a cube $Q \in \dyadic_t$ and $u \in L^\infty(w;\C^{n+1})$ with $\|u\|_\infty=1$. Then
write $u= u_0+ u_1+u_2+\ldots$ where $u_0=u$ on $2Q$ and $0$ elsewhere and if $j\ge 1$, $u_j=u$ on $2^{j+1}Q \setminus 2^{j}Q$ and $0$ elsewhere. Then apply $Q_t^B$  and use Lemma \ref{lem:odd} for each term $Q_t^B u_j$ with $N$ large enough  and sum  to obtain
$$
   \barint_{\hspace{-6pt}Q} |(Q_t^Bu)(y)|^2 \, dw(y) \le
  C.
$$
If we do this for  the constant functions with values describing an orthonormal basis of $\C^{n+1}$ and sum, we obtain an upper bound for the desired average of $\gamma_t$.
Next,  for a function $u \in \mH$ and $Q\in \Delta_{t}$, as $E_{t}u$ is constant on $Q$,
\begin{align*}
\label{}
   \int_Q \left|\gamma_t (y)E_{t}u(y)\right|^2 \, dw(y)&\le \int_Q \left|\gamma_t (y)\right|^2 \, dw(y) \times  \barint_{\hspace{-6pt}Q} \left|E_{t}u(y)\right|^2 \, dw(y) \\
   &\lesssim    \int_Q \left|E_{t}u(y)\right|^2 \, dw(y).
   \end{align*}  Thus
$$
\|\gamma_t E_{t}u\|^2  \lesssim  \sum_{Q\in \dyadic_t}  \int_Q \left|E_{t}u(y)\right|^2 \, dw(y) = \|E_{t}u\|^2 \lesssim \|u\|^2.
$$
\end{proof}

As $Q_t^B$ vanishes on $ \nul(DB)$, it is enough to prove the quadratic estimate (\ref{eq:sfBD}) for $v \in \clos{\ran(DB)}=\clos{\ran(D)}$.
 Our principal part approximation reads as follows.

\begin{prop}  \label{lem:ppa} We have
 \begin{equation}\label{eq:ppa1}
  \qe{Q_t^B v-\gamma_t E_{t} v}
  \lesssim \|v\|^2,  \quad v\in \clos{\ran(D)}.
\end{equation}

\end{prop}

The function  from $\R^{n+1}_+$ to $ \R_+$ defined by $ (x,t) \mapsto |\gamma_t(x)|^2$, is
a {\em weighted dyadic Carleson function} if there exists $C<\infty$
such that
$$
  \iint_{\carl{Q}} |\gamma_t(x)|^2 \, \frac{dw(x)dt}t \le C^2 w(Q)
$$
for all dyadic cubes $Q\subseteq\R^n$.
Here $\carl{Q}:= Q\times (0,l(Q)]$ is the Carleson box over $Q$. We define the dyadic Carleson norm $\|\gamma_t\|_C$ to be the smallest
constant $C$ for which this inequality holds.  The form of Carleson's lemma that we need and will apply componentwise is  the following (see \cite{AT}, p.168 and references therein).
\begin{lem}  \label{lem:Carleson} For all $u\in \mH$,
$$
  \qe{\gamma_t E_{t}u} \lesssim \|\gamma_t\|_C^2 \| \sup_{t>0}|E_{t}u|\|^2 \lesssim \|\gamma_t\|_C^2 \| u\|^2.
$$
\end{lem}

\begin{cor}\label{cor:cor1} If $|\gamma_{t}(x)|^2$ is a weighted dyadic Carleson function,  then Theorem \ref{th:main} holds.
\end{cor}

This corollary clearly  follows from the above lemma and \eqref{eq:ppa1}.

\subsection{Proof of the principal part approximation}

We begin the proof of the principal part approximation \eqref{eq:ppa1} with some further notation.
Define
$$P_{t}=\begin{bmatrix}
        (I-t^2\Delta_{w})^{-1} &  0  \\
        0  & (I-t^2\Delta)^{-1} I_{\C^n}
     \end{bmatrix}.
     $$

           Here $\Delta_{w}=\divv_{w}\nabla$ is the  negative   self-adjoint operator  on $L^2(w)$ defined in Lemma \ref{lem:gradient} while $\Delta$ is the usual  negative  Laplacian  on $L^2(dx)$.  From Lemma \ref{lem:approx},
       the convolution operator  $(I-t^2\Delta)^{-1}$ is bounded on $L^2(w)$ uniformly with respect to $t>0$: It is indeed classical that $(I-t^2\Delta)^{-1}$ is the convolution with $t^{-n}G_{2}(x/t)$ where the Bessel potential $G_{2}$ is integrable and radially decreasing (see \cite[Chapter 6]{Gra}).  Thus, $P_{t}$ is uniformly bounded on $\mH$.

\begin{lem}\label{whatisneeded}   For all $ v \in \clos{\ran(D)}$, one has
\begin{equation}  \label{firsttermprop}
     \qe{Q_t^B(I-P_t) v}\lesssim \|v\|^2.
\end{equation}
\end{lem}

\begin{proof}\label{proof:thm}   The key  point is that $ \clos{\ran(D)}$ is preserved by $P_{t}$. This follows from the characterization of  $\clos{\ran(D)}$ in Proposition \ref{prop:D}  as $\begin{bmatrix}
       L^2(w) \\
     \mR(L^2(w))
     \end{bmatrix}$, since  the Riesz transforms and $(I-t^2\Delta)^{-1}$ commute and $\mR(- \Delta)^{1/2}=\nabla$.
     Write $v=\begin{bmatrix}
       f \\
     \mR g
     \end{bmatrix}$, with $f, g\in L^2(w)$.  Then,
     \begin{align*}
     (I-P_{t})v&=\begin{bmatrix}
       -t^2\Delta_{w} (I-t^2\Delta_{w})^{-1}f   \\
        -t^2\Delta (I-t^2\Delta)^{-1} \mR g
     \end{bmatrix}\\
     & = \begin{bmatrix}
       -t^2\divv_{w} \nabla (I-t^2\Delta_{w})^{-1}f   \\
        t^2\nabla (- \Delta)^{1/2} (I-t^2\Delta)^{-1}  g
     \end{bmatrix}\\
     &  =  - tD  \begin{bmatrix}
               t (- \Delta)^{1/2} (I-t^2\Delta)^{-1}  g
               \\
               t \nabla (I-t^2\Delta_{w})^{-1}f
\end{bmatrix} .
\end{align*}
Thus using that $tQ_{t}^BD$ is uniformly bounded on $\mH$, it suffices to prove
$$
\qe{t \nabla (I-t^2\Delta_{w})^{-1}f}\lesssim \|f\|^2
$$
and
$$\qe{t (- \Delta)^{1/2} (I-t^2\Delta)^{-1}  g} \lesssim \|g\|^2.
$$
The first estimate is a simple consequence of the construction of $-\Delta_{w}$ as the self-adjoint operator $\nabla^*\nabla$ on $L^2(w)$.  Indeed, $\|\nabla (I-t^2\Delta_{w})^{-1}f\|= \| (-\Delta_{w})^{1/2} (I-t^2\Delta_{w})^{-1}f\|$ and one concludes using the spectral theorem for the self-adjoint operator $\Delta_{w}$ that $\qe{t \nabla (I-t^2\Delta_{w})^{-1}f }= c \|f\|^2$ with $c= \int_{0}^\infty t(1+t^2)^{-1} \frac{dt}{t}$.
The second estimate is a consequence of Lemma \ref{lem:LPw} as soon as the conditions for the function $\psi$ defined on the Fourier transform side by $\hat\psi(\xi)= |\xi|(1+|\xi|^2)^{-1}$ have been checked.  \end{proof}

 \begin{lem}\label{lem:ppalem2} For all $v \in \clos{\ran(D)}$, one has
 $$
 \qe{Q_t^BP_tv -\gamma_tE_{t}P_tv } \lesssim  \|v\|^2.
 $$
 \end{lem}

  \begin{proof}
  We remark that for $t>0$ fixed and $x\in \R^n$, we have
$$(Q_t^BP_tv -\gamma_tE_{t}P_t v)(x)= Q_t^B (u- u_Q ) (x),
$$
where $u=P_tv$,  $Q$ is the unique dyadic cube in $\dyadic_t$ containing $x$ and
$u_{Q}$ is the value of $E_{t}u$ on $Q$.
Define $C_0(Q)=2Q$ and $C_j(Q)=2^{j+1}Q\setminus 2^jQ$ if $j\in \N^*$. Then,
\begin{align*}
\|Q_t^BP_t v-\gamma_tE_{t}P_t v  \|^2 &=  \sum_{Q\in \dyadic_t} \int_Q | Q_t^B(u- u_{Q})|^2 \, dw
\\
& \le \sum_{Q\in \dyadic_t} \left(\sum_{j\ge 0} \bigg(\int_Q | Q_t^B ({\bf 1}_{C_j(Q)} (u -u_{Q}))|^2\, dw \bigg)^{1/2}\right)^{2}
\\
&
\lesssim  \sum_{Q\in \dyadic_t} \left(\sum_{j\ge 0} 2^{-jN} \bigg(\int_{C_j(Q)}|u-u_{Q} |^2\, dw \bigg)^{1/2}\right)^{2}
\\
&
\lesssim  \sum_{Q\in \dyadic_t} \sum_{j\ge 0}  2^{-jN} \int_{C_j(Q)}  |u - u_{Q} |^2\, dw
\\
&
\lesssim  \sum_{Q\in \dyadic_t} \sum_{j\ge 0}  2^{-jN} 2^{2j} \ell(Q)^2 2^{jd}\int_{2^{j+1}Q} |\nabla  u |^2
\, dw \\
&
\lesssim   t^2 \sum_{j\ge 0}  2^{-jN} 2^{2j} 2^{jd} 2^{jn} \int_{\R^n} |\nabla u   |^2\, dw.
\\
&
\lesssim   t^2 \|\nabla  u \|^2.
\end{align*}
We  used the Minkowski inequality on the second line, off-diagonal decay  on the third,  Cauchy--Schwarz inequality on the fourth,
Poincar\'e inequality on the fifth  (recalling that one can take  the average  with respect to either  $dx$ or $dw$), and a telescoping argument which produces the doubling exponent $d$ of $w$,
the covering inequality $ \sum_{Q\in \dyadic_t}  {\bf 1}_{2^{j+1}Q} \lesssim 2^{jn}$ and $\ell(Q)\sim t$  on the sixth. Finally, we choose $N> n+d+ 2$ in the last.
  Hence
$$\qe{Q_t^BP_tv-\gamma_t E_{t} P_t v}
  \lesssim \qe{t \nabla  P_t v}.
  $$

 For the scalar part of $v$, we have to control the weighted quadratic estimate for $t\nabla (I-t^2\Delta_{w})^{-1}v_{\no}$ which we have seen already. Using  $v\in \clos{\ran(D)}$,  the tangential part $v_{\ta}$ is of the form $\mR g$ for some $g\in L^2(w)$. Hence we have to control the  quadratic estimate of $t\nabla  (I-t^2\Delta)^{-1}R_{j} g=R_{j} \mR t  (-\Delta)^{1/2}(I-t^2\Delta)^{-1}g$ for $j=1, \ldots, n$. We can eliminate $R_{j}$ and $\mR$ as the Riesz transforms are bounded on $L^2(w)$ and the weighted Littlewood-Paley estimate  for $t  (-\Delta)^{1/2}(I-t^2\Delta)^{-1}$ has been already seen.   \end{proof}

  \begin{lem}\label{lem:mean1} There are constants $C<\infty$ and $\tau_{1}\in(0,1)$ such that for all $f\in \dom(\divv_{w})$ and all  dyadic cubes $Q$,
  $$
\left| \barint_{\hspace{-6pt}Q} \divv_{w}f \, dw \right| \le   \frac{C}{\ell(Q)^{\tau_{1}}}
\left( \barint_{\hspace{-6pt}Q} |\divv_{w}f|^2 \, dw \right)^{\frac{1-\tau_{1}}2}\left( \barint_{\hspace{-6pt}Q} |f|^2 \, dw \right)^{\frac{\tau_{1}}2}.
$$
  \end{lem}

\begin{proof} Observe that if $f$ has support contained in $Q$, then $\int_{Q}  \divv_{w}f \, dw=0$.   Thus this lemma  follow from \cite{Ban}. Here is a simple proof in our situation.   Let $A=\left( \barint_{\hspace{-2pt}Q} |\divv_{w}f|^2 \, dw \right)^{1/2}$ and $B=\left( \barint_{\hspace{-2pt}Q} |f|^2 \, dw \right)^{1/2}$. If $B\ge A\ell$ a simple application of the Cauchy-Schwarz inequality gives the result.

Assume next that $B<A\ell$.
Let  $\varphi\colon \R^n\to [0,1]$ be a smooth function which equals to  1 on $(1-t)Q$, 0 on $Q^c$
with $\|\nabla\varphi\|_\infty \le C(t\ell)^{-1}$ for some $t\in (0,1)$ to be chosen. We can write
$$
\barint_{\hspace{-6pt}Q} \divv_{w}f\, dw
= \barint_{\hspace{-6pt}Q} \varphi\, \divv_{w}f \, dw + \barint_{\hspace{-6pt}Q} (1-\varphi)\divv_{w}f \, dw
=I+II.
$$
Using that $-\divv_{w}$ is the adjoint of $\nabla$ (writing the integral as an integral on $\R^n$ thanks to the support of $\varphi$),
$$| I | = \left|\barint_{\hspace{-6pt}Q} \nabla \varphi f\, dw\right| \le C(t\ell)^{-1} B.$$
For $II$,   using Cauchy-Schwarz inequality  and  the right hand side of  \eqref{eq:ainfty} with $E$ being  the support of $1-\varphi$ in $Q$,
$$
| II |  \le A  \frac{w(Q\setminus (1-t)Q)^{1/2}}{w(Q)^{1/2}} \le C_{w}  A t^{\sigma/2}.
$$
Hence, choosing $t^{1+\sigma/2}=B/A\ell$, we obtain the inequality with $\tau_{1}=\frac{\sigma/2}{1+\sigma/2}$.
\end{proof}

 \begin{lem}\label{lem:mean2} There are constants $C<\infty$ and $\tau_{2}\in(0,1)$ such that for all $g\in \dom(\nabla)$ and all dyadic cubes $Q$,
  $$
\left| \barint_{\hspace{-6pt}Q} \nabla g \, dx \right| \le   \frac{C}{\ell(Q)^{\tau_{2}}}
\left( \barint_{\hspace{-6pt}Q} |\nabla g|^2 \, dw \right)^{\frac{1-\tau_{2}}2}\left( \barint_{\hspace{-6pt}Q} |g|^2 \, dw \right)^{\frac{\tau_{2}}2}.
$$
  \end{lem}

\begin{proof} The proof follows a similar pattern to that of Lemma \ref{lem:mean1}.     Let $A=\left( \barint_{\hspace{-2pt}Q} |\nabla g|^2 \, dw \right)^{1/2}$ and $B=\left( \barint_{\hspace{-2pt}Q} |g|^2 \, dw \right)^{1/2}$. If $B\ge A\ell$ a simple application of \eqref{eq:dxdwAp} with $p=2$ gives the result.

Assume next that $B<A\ell$.
Let  $\varphi\colon \R^n\to [0,1]$ be a smooth function which equals to  1 on $(1-t)Q$, 0 on $Q^c$
with $\|\nabla\varphi\|_\infty \le C(t\ell)^{-1}$ for some $t\in (0,1)$ to be chosen. We can write
$$
\barint_{\hspace{-6pt}Q}\nabla g \, dx
= \barint_{\hspace{-6pt}Q} \varphi \nabla g\, dx + \barint_{\hspace{-6pt}Q} (1-\varphi)\nabla g \, dx
=I+II.
$$
Interpreting $I$ as a distribution-test function pairing, we may integrate by parts coordinatewise,  and using that $w\in A_{2}$ we find
$$| I | = \left|\barint_{\hspace{-6pt}Q}  g \nabla \varphi \, dx\right| \le C(t\ell)^{-1} \barint_{\hspace{-6pt}Q}  |g|\, dx \le C[w]_{A_{2}}(t\ell)^{-1} B.$$
For $II$, pick $p>1$ such that  $w\in A_{2/p}$ and use H\"older's inequality with conjugate exponents $p$ and $p'$, the support  properties of $1-\varphi$, and the fact that  $w\in A_{2/p}$, to conclude that
$$
| II |  \lesssim \left( \barint_{\hspace{-6pt}Q} |\nabla g|^p \, dx \right)^{1/p} t^{1/p'} \le [w]_{A_{2/p}}^{1/p}  A t^{1/p'}.
$$
Hence, choosing $t^{1+1/p'}=B/A\ell$, we obtain the inequality with $\tau_{2}=\frac{1/p'}{1+1/p'}$.
\end{proof}

\begin{lem}\label{lem:ppalem3} For all $u\in \mH=L^2(w;\C^{n+1})$,
$$\qe{\gamma_tE_{t}(P_t-I)u} \lesssim \|u\|^2.$$
\end{lem}

 \begin{proof} {It follows from Lemma \ref{lem:gammat} that
 $\|\gamma_tE_{t}(P_t-I)u\| \lesssim \|E_{t}(P_t-I)u\|$. Given the definitions of $E_{t}$ and $P_{t}$, the lemma reduces to  the scalar inequalities
 $$\qe{A_t^w((I-t^2\Delta_{w})^{-1}-I)f} \lesssim \|f\|^2$$
 and
  $$\qe{A_t((I-t^2\Delta)^{-1}-I)f} \lesssim \|f\|^2.$$

 For the first one, we follow \cite{AKMc} with a minor simplification. Let $Q_{s}^w= s^2\Delta_{w}e^{s^2\Delta_{w}}$. Then for $f\in L^2(w)$ we have by the spectral theorem,  $f= 8  \int_{0}^\infty (Q_{s}^w)^2f\, \frac {ds}s$
 and also $\|f\|^2= 8 \int_{0}^\infty \|Q_{s}^w f\|^2 \, \frac {ds}s$.
 By Schur's lemma,  it is enough to show that  the operator norm of $A_t^w((I-t^2\Delta_{w})^{-1}-I)Q_{s}^w$ in $L^2(w)$ is bounded by $h(s/t)$ with $h\ge 0$ and $\int_{0}^\infty h(u)\, \frac{du}u<\infty$. We shall find $h(u)= C \inf (u^{\tau_{1}}, u^{-2})$.

 If $t<s$,
 \begin{align*}
\label{}
   \|A_{t}^w((I-t^2\Delta_{w})^{-1}-I)Q_{s}^wf\| &\lesssim \|((I-t^2\Delta_{w})^{-1}-I)s^2\Delta_{w}e^{s^2\Delta_{w}}f\| \\
   &\lesssim C(t/s)^2\|f\|.
    \end{align*}

 If $s<t$,
 \begin{align*}
\label{}
   \|A_{t}^w(I-t^2\Delta_{w})^{-1}Q_{s}^wf\|&\lesssim \|(I-t^2\Delta_{w})^{-1}s^2\Delta_{w}e^{s^2\Delta_{w}}f\|
   \\
   &
   \lesssim C(s/t)^2\|f\|. \end{align*}

 Thus, it remains to study the operator norm of $A_{t}^wQ_{s}^w$ for $s<t$. For this, we remark
 that $A_{t}^wQ_{s}^wf(x) $ is an average with respect to $dw$ of $ \divv_{w}(sg)$ with $g=-s \nabla e^{s^2\Delta_{w}}f$ so that we can use Lemma \ref{lem:mean1}.
 Thus,
 \begin{align*}
 \|A_{t}^wQ_{s}^wf\|^2&= \sum_{Q\in \Delta_{t}}  w(Q) \left| \barint_{\hspace{-6pt}Q} \divv_{w}(sg)\, dw \right|^2
 \\
 &
 \lesssim  \frac {1}{t^{2\tau_{1}}} \sum_{Q\in \Delta_{t}}  w(Q)  \left( \barint_{\hspace{-6pt}Q} |\divv_{w}(sg)|^2 \, dw \right)^{{1-\tau_{1}}}\left( \barint_{\hspace{-6pt}Q} |sg|^2 \, dw \right)^{{\tau_{1}}}
 \\
& = \frac {s^{2\tau_{1}}}{t^{2\tau_{1}}} \sum_{Q\in \Delta_{t}}    \left( \int_{Q} |Q_{s}^wf|^2 \, dw \right)^{{1-\tau_{1}}}\left( \int_{Q} |s \nabla e^{s^2\Delta_{w}}f|^2 \, dw \right)^{{\tau_{1}}}
\\
 &\le  \frac {s^{2\tau_{1}}}{t^{2\tau_{1}}} \|Q_{s}^wf\|^{2(1-\tau_{1})}  \|s \nabla e^{s^2\Delta_{w}}f\|^{2\tau_{1}}
 \\
 &
 \lesssim  \frac {s^{2\tau_{1}}}{t^{2\tau_{1}}} \|f\|^2
 \end{align*}
 where we used H\"older's inequality for the sum.

 The proof of the second inequality is as follows. Setting $Q_{s}= s^2\Delta e^{s^2\Delta }$, as before  we have that  the operator norm  of
 $A_t((I-t^2\Delta)^{-1}-I)Q_{s}$ on $L^2(dx)$ is bounded by $h(s/t)$ with $h$ similar to  above. As this operator is also bounded on $L^2(w)$ uniformly in $s$ and $t$ for all $w\in A_{2}$ (for the convolution operators,  it follows from Lemma \ref{lem:approx}  and for  $A_{t}$ this has been seen before), its operator norm on $L^2(w)$ is bounded by $h(s/t)^\theta$ for some power $\theta>0$ depending only on  $w$  using Lemma \ref{lem:DRF}. Thus, we can use  the fact that  the integral $f=  8  \int_{0}^\infty (Q_{s})^2f\, \frac {ds}s$ converges in $L^2(w)$ from Corollary \ref{cor:Calderon}, Schur's Lemma and that $ \int_{0}^\infty \|Q_{s} f\|^2 \, \frac {ds}s \lesssim \|f\|^2$ from Lemma \ref{lem:LPw}.
 }
 \end{proof}

\begin{proof}[Proof of  Proposition \ref{lem:ppa}]
It is enough to write
$$
Q_t^B v-\gamma_t E_{t} v= Q_t^B(I-P_{t}) v +  (Q_t^BP_{t}v -\gamma_t E_{t}P_{t} v) +  \gamma_t E_{t}(P_{t}-I) v
$$
and to use successively Lemma \ref{whatisneeded}, Lemma \ref{lem:ppalem2} and Lemma \ref{lem:ppalem3}.
\end{proof}

\subsection{Preamble to the Carleson measure estimate}

We are now ready to prove that $|\gamma_t(x)|^2$ is a weighted dyadic Carleson function and so, via Proposition \ref{lem:ppa}, complete the proof of Theorem \ref{th:main}. The first step towards this is a compactness argument. As was seen in the solution of the Kato square root problem (\cite{AHLMcT}), the application of a stopping-time argument was made possible by restricting $\gamma_t(x)$ so that, once normalised, it is close to a fixed element in the unit sphere of $\mathcal{L}(\C^{1+n})$, the set of bounded linear transformations on $\C^{1+n}$. We will make use of the same stopping-time argument, but also require a second stopping-time related to the oscillation of the weight, and with it comes a second compactness argument which restricts our attention to Whitney boxes on which the average of the weight is close to that of the top cube.

 A convenient way to define the  Whitney box $W_{Q}$ associated to a given  dyadic cube $Q$ is
$W_{Q}=\{ (x,t)\, ; \, x\in Q, Q\in \Delta_{t}\} $. With our definition $\widehat Q$ is the union all $W_{Q'}$ for which $Q'$ is a dyadic subcube of $Q$.

Consider the compact unit sphere in $\mathcal{L}(\C^{1+n})$ and the compact interval $[0,c_0]$, where $c_0$ is as in \eqref{reversejensen}. For each $\nu \in \mathcal{L}(\C^{1+n})$ such that $|\nu| = 1$, $\tau \in [0,c_0]$ and $\sigma_{1},\sigma_2 > 0$, define   $G_{\tau,\sigma_2}$ as the union of those Whitney boxes $W_{Q}$ for which $|\ln(w_Q) - (\ln w)_Q - \tau| < \sigma_2$,
 and
\begin{equation}\label{gammatilde}
\widetilde{\gamma}_t(x)=
\begin{cases} \gamma_t(x) & \text{if $\gamma_t(x)\ne 0$,  $\left|\frac{\gamma_t(x)}{|\gamma_t(x)|} - \nu \right| \leq \sigma_1$ and $(x,t) \in G_{\tau,\sigma_2}$,}
\\
0 &\text{otherwise.}
\end{cases}
\end{equation}
 We recall the notation $B^w(Q)$ from Section \ref{coronasec} and set
$$\Omega^w(Q) = \carl{Q}   \setminus \cup_{R \in B^w(Q)} \carl{R}.
$$

\begin{lem} \label{lemma5}
Suppose that we can show
\[
 K=  \sup_{\nu, \tau}\sup_{Q \in \dyadic} \frac{1}{w(Q)} \iint_{\Omega^w(Q)} |\widetilde{\gamma}_t(x)|^2\,  \frac{dw(x)dt}{t} < \infty
\]
for some choice of parameters $\sigma_1$ and $\sigma_2$ depending only on $\|B\|_{\infty}$, $\kappa$, $[w]_{A_2}$ and $n$. Then $|\gamma_t(x)|^2$ is a weighted dyadic Carleson function.
\end{lem}

\begin{proof}  Fix $\sigma_{1}$ and $\sigma_{2}$ so that the hypothesis applies.  Let $Q\in \Delta$.  Observe  that the sets $\Omega^w(R)$ form a partition of $\carl{Q}$ when $R$ runs over elements of $ B^w_*(Q) \cup \{Q\}$.  Thus,
by the hypothesis and Proposition \ref{prop4},
\begin{align*}
\iint_{\carl{Q}} |\widetilde{\gamma}_t(x)|^2\,  \frac{dw(x)dt}{t} & = \sum_{R \in B^w_*(Q) \cup \{Q\}} \iint_{\Omega^w(R)} |\widetilde{\gamma}_t(x)|^2\,  \frac{dw(x)dt}{t} \\
& \leq  K   \sum_{R \in B^w_*(Q) \cup \{Q\}}  w(R) \leq  \frac{KC}{\sigma^2_{w}}  w(Q)  + K w(Q).
\end{align*}
 By the compactness of the unit sphere in $\mathcal{L}(\C^{1+n})$ and the interval $[0,c_0]$, there exist a finite index set $A \subseteq \N$ and, for each $j \in A$, choices of $\nu_j$ and $\tau_j$ such that $|\gamma_t(x)|^2 \leq \sum_{j \in A} |\widetilde{\gamma}^j_t(x)|^2$, where $\widetilde{\gamma}^j_t(x) = \widetilde{\gamma}_t(x)$ with the choice $\nu = \nu_j$ and $\tau = \tau_j$.
This completes the proof.
\end{proof}

\subsection{Stopping-time arguments for test functions}

We fix an arbitrary vector $\xi$ in the unit sphere of $\C^{1+n}$.  For any  $Q_1 \in \dyadic$ and   $\sigma_{3}>0$ to be chosen, define a test function
\begin{equation}
\label{testf}
f^\xi_{Q_1} := \Big(I + \big(\sigma_3\ell(Q_{1})DB\big)^2\Big)^{-1}(\etaq\xi)  = P^B_{\sigma_3\ell(Q_{1})}(\etaq\xi),
\end{equation}
where   $\etaq$ is the indicator of $Q_{1}$.
    Note that $\|f^\xi_{Q_1}\|^2 \lesssim w(Q_{1})$ and $\|\sigma_{3}\ell(Q_{1})DBf^\xi_{Q_1}\|^2\lesssim w(Q_{1})$ with uniform implicit constants with respect to  $|\xi|=1$, $\sigma_{3}>0$ and $Q_{1}$,   as can be seen using the uniform  boundedness in $t>0$ of $Q^B_{t}$ and $P^B_{t}$.

\begin{lem} \label{lemma6}
There exist a constant $c$ depending only on $\|B\|_{\infty}$, $\kappa$, $[w]_{A_2}$ and $n$, and  a constant $\delta > 0$ depending only on $[w]_{A_2}$, such that for all such $\xi$, $Q_{1}$ and $\sigma_{3}$,
\[
|E_{Q_1}(f^\xi_{Q_1}) - \xi| \leq c\sigma_3^{\delta}.
\]
\end{lem}

\begin{proof}
We have  $E_{Q_1}(f^\xi_{Q_1}) - \xi = E_{Q_1}D u$ with
\[
u: = - (\sigma_3 \ell(Q_1))^2BDB \Big(I + \big(\sigma_3\ell(Q_{1})DB\big)^2\Big)^{-1}(\etaq\xi)
\]
 and notice that $E_{Q_1}D$ acts on  $u =  \begin{bmatrix}
u_1 \\
u_2
\end{bmatrix}$ componentwise by averaging  $\divv_w u_2$ with respect to $dw$ and $\nabla u_1$ with respect to $dx$.
Lemma \ref{lem:mean1} says
\begin{align*}
& \left| \barint_{\hspace{-6pt}Q_1} \divv_{w}u_2 \, dw \right| \le   \frac{C}{\ell(Q_1)^{\tau_{1}}}
\left( \barint_{\hspace{-6pt}Q_1} |\divv_{w}u_2|^2 \, dw \right)^{\frac{1-\tau_{1}}2}\left( \barint_{\hspace{-6pt}Q_1} |u_2|^2 \, dw \right)^{\frac{\tau_{1}}2} \\
& \leq   C\sigma_3^{\tau_1}
\left( \barint_{\hspace{-6pt}Q_1} | \xi - f^\xi_{Q_1}  |^2 \, dw \right)^{\frac{1-\tau_{1}}2}\left( \barint_{\hspace{-6pt}Q_1} | \sigma_{3}\ell(Q_{1})DBf^\xi_{Q_1}|^2 \, dw \right)^{\frac{\tau_{1}}2} \\
& \leq   C\sigma_3^{\tau_1}
\end{align*}
and Lemma \ref{lem:mean2} says
\begin{align*}
& \left| \barint_{\hspace{-6pt}Q} \nabla u_1 \, dx \right| \le   \frac{C}{\ell(Q)^{\tau_{2}}}
\left( \barint_{\hspace{-6pt}Q} |\nabla u_1|^2 \, dw \right)^{\frac{1-\tau_{2}}2}\left( \barint_{\hspace{-6pt}Q} |u_1|^2 \, dw \right)^{\frac{\tau_{2}}2} \\
& \leq   C\sigma_3^{\tau_2}
\left( \barint_{\hspace{-6pt}Q_1} |\xi - f^\xi_{Q_1} |^2 \, dw \right)^{\frac{1-\tau_{2}}2}\left( \barint_{\hspace{-6pt}Q_1} | \sigma_{3}\ell(Q_{1})DBf^\xi_{Q_1}  |^2 \, dw \right)^{\frac{\tau_{2}}2} \\
& \leq   C\sigma_3^{\tau_2}.
\end{align*}
So taking $\delta = \min(\tau_1,\tau_2)$ completes the proof.
\end{proof}

Recall that  $w_{Q}= \barint_{\hspace{-2pt}Q} w\, dx$ and similarly for $(\ln w)_{Q}$.

\begin{lem} \label{lemma7}
Fix $\tau\in [0,c_{0}]$ and $\xi$ in the unit sphere of $\C^{1+n}$. Let
\[
S_{Q_1}^\tau=\begin{bmatrix}
        w_{Q_1} e^{-\tau - (\ln w)_{Q_1}} &  0  \\
        0  & I
     \end{bmatrix},
\]
and define the collection of `bad' cubes $B^{\tau,\xi}(Q_1)$ to be the set of maximal $Q' \in \dyadic$ such that $Q' \subseteq Q_1$ and
\[
\mbox{either $|E_{Q'}(f^\xi_{Q_1})| > \frac{1}{\sigma_4}$ or $ \re  \left(S_{Q_1}^\tau(\xi),\begin{bmatrix}
        w_{Q'}/w_{Q_1} &  0  \\
        0  & I
     \end{bmatrix}E_{Q'}(f^\xi_{Q_1})\right) < \sigma_5$.}
\]
We can then choose positive $\sigma_3$, $\sigma_4$ and $\sigma_5$  depending only on $\|B\|_{\infty}$, $\kappa$, $[w]_{A_2}$ and $n$, in particular independently on $\tau,\xi, Q_{1}$, so that
\begin{equation} \label{geometric}
\sum_{R \in B^{\tau,\xi}(Q_1)} w(R) \leq (1-\sigma_6)w(Q_1),
\end{equation}
with $0< \sigma_{6}\le 1$.
\end{lem}

\begin{proof}
There are two sets of cubes to consider. The first is the set of those   maximal  $Q'$ for which
\begin{equation} \label{stopone}
  \re  \left(S_{Q_1}^\tau(\xi),\begin{bmatrix}
        w_{Q'}/w_{Q_1} &  0  \\
        0  & I
     \end{bmatrix}E_{Q'}(f^\xi_{Q_1})\right) <  \sigma_5.
\end{equation}
By \eqref{reversejensen}, we know that
\[
-c_0 \leq \ln(w_{Q_1}) - \tau - (\ln w)_{Q_1} \leq c_0
\]
so $S_{Q_1}^\tau$ is a constant  self-adjoint  matrix with  $e^{-c_{0}} I \le S_{Q_1}^\tau  \le e^{c_{0}} I$. Applying Lemma \ref{lemma6},
\begin{align}
 \re   \big(S_{Q_1}^\tau(\xi),E_{Q_1}(f^\xi_{Q_1})\big) & = \big(S_{Q_1}^\tau(\xi),\xi\big) +   \re  \big(S_{Q_1}^\tau(\xi), E_{Q_1}(f^\xi_{Q_1})-\xi\big) \nonumber \\
& \geq e^{-c_{0}} - ce^{c_{0}}\sigma_3^{\delta} \geq \frac 1 2 e^{-c_{0}}, \label{sigma3}
\end{align}
on choosing $\sigma_3$ so that $2c\sigma_{3}^\delta \le e^{-2c_{0}}$. Consequently, setting $G = Q_1 \setminus (\cup Q')$   and $f^\xi_{Q_1}= \begin{bmatrix}
       f_{1}   \\ f_{2}
        \end{bmatrix}
        $,
 we have
 \begin{align*}
\label{}
    E_{Q_1}(f^\xi_{Q_1})&= \begin{bmatrix}
        \frac{1}{w(Q_1)} \int_{Q_{1}} f_1\, dw  \\
        \frac{1}{|Q_1|} \int_{Q_{1}} f_2\,  dx
     \end{bmatrix}  \\
    &
    =\sum_{Q'} \frac{|Q'|}{|Q_1|}  \begin{bmatrix}
        w_{Q'}/w_{Q_{1}} & 0 \\
        0 & I
     \end{bmatrix}E_{Q'}(f^\xi_{Q_1}) +   \begin{bmatrix}
        \frac{1}{w(Q_1)} \int_G f_1\, dw  \\
        \frac{1}{|Q_1|} \int_G f_2\,  dx
     \end{bmatrix} ,
     \end{align*}
where the subcubes $Q'$ are those of \eqref{stopone}.     Using \eqref{stopone} and \eqref{sigma3}, we obtain
        \begin{align*}
\frac 1 2 e^{-c_{0}}
&
 \leq   \re  \big(S_{Q_1}^\tau(\xi),E_{Q_1}(f^\xi_{Q_1})\big) \\
& \leq \sigma_5 \sum_{Q'} \frac{|Q'|}{|Q_1|} +   \re  \left(S_{Q_1}^\tau(\xi),\begin{bmatrix}
        \frac{1}{w(Q_1)} \int_G f_1\,  dw  \\
        \frac{1}{|Q_1|} \int_G f_2\,  dx
     \end{bmatrix}\right) \\
& \leq \sigma_5 +  \re  \left(S_{Q_1}^\tau(\xi),\begin{bmatrix}
        \frac{1}{w(Q_1)} \int_G f_1\,  dw  \\
        \frac{1}{|Q_1|} \int_G f_2\,  dx
     \end{bmatrix}\right)
\end{align*}
and,  using the estimate $\int_{Q_{1}} |f^\xi_{Q_1}|^2\, dw \lesssim w(Q_{1}) $ and again the $A_{2}$ condition for $w$,
\begin{align*}
& \left|\left(S_{Q_1}^\tau(\xi),\begin{bmatrix}
        \frac{1}{w(Q_1)} \int_G f_1\, dw  \\
        \frac{1}{|Q_1|} \int_G f_2 \, dx
     \end{bmatrix}\right)\right| \\
& \leq e^{c_{0}}\left(\frac{w(G)}{w(Q_1)^2}\left(\int_G |f_1|^2 dw\right) + \frac{w^{-1}(G)}{|Q_1|^2}\left(\int_G |f_2|^2 dw\right)\right)^{1/2} \\
& \lesssim \left(\frac{w(G)}{w(Q_1)}\right)^{1/2} + \frac{(w^{-1}(G))^{1/2}w(Q_1)^{1/2}(w^{-1}(Q_1))^{1/2}}{(w^{-1}(Q_1))^{1/2}|Q_1|} \\
& \lesssim \left(\frac{w(G)}{w(Q_1)}\right)^{1/2} + \left(\frac{w^{-1}(G)}{w^{-1}(Q_1)}\right)^{1/2} \lesssim \left(\frac{w(G)}{w(Q_1)}\right)^{\theta}.
\end{align*}
for some $\theta>0$ by \eqref{eq:ainfty} applied to $w^{-1}$ and $w$.
 Therefore, for a small enough choice of $\sigma_5$, we have that
\[
\left(\frac{w(G)}{w(Q_1)}\right)^{\theta} \gtrsim 1
\]
which implies that $w(G) \geq 2\sigma_6 w(Q_1)$ for some small $\sigma_6 > 0$ and so
\begin{equation} \label{firsthalf}
\sum_{Q'} w(Q') \leq (1-2\sigma_6) w(Q_1),
\end{equation}
where the sum is taken over those cubes $Q'$ which satisfy \eqref{stopone}.

Now we consider the set of maximal dyadic subcubes $Q'$ of $Q_{1}$ for which
\begin{equation} \label{stoptwo}
|E_{Q'}(f^\xi_{Q_1})| > \frac{1}{\sigma_4}
.\end{equation}
 Then $w(Q') \leq  \sigma_4^2  \int_{Q'} |f^\xi_{Q_1}|^2 dw$ and
\[
\sum_{Q'} w(Q') \leq \sum_{Q'} \sigma_4^2 \int_{Q'} |f^\xi_{Q_1}|^2 dw \leq \sigma_4^2  \int_{Q_1} |f^\xi_{Q_1}|^2 dw \lesssim \sigma_4^2\,   w(Q_1)
\]
where the sum is now over those cubes $Q'$ which satisfy \eqref{stoptwo}.
So we can choose $\sigma_4$ so small that
\begin{equation} \label{secondhalf}
\sum_{Q'} w(Q') \leq \sigma_6\,  w(Q_1).
\end{equation}
Combining \eqref{firsthalf} and \eqref{secondhalf} proves the lemma.
\end{proof}

\subsection{Conclusion of the Carleson measure estimate} Consider $\tilde{\gamma}_t(x)$ depending on  $\nu$ and $ \tau$ as defined in \eqref{gammatilde}.
 Associate to $\nu$ a vector  $\xi \in \C^{1+n}$  such that $|\nu(\xi)|=1$ and $|\xi|=1$. Such a $\xi$ may not be uniquely  defined but we pick one.
For a cube $Q_1 \in \dyadic$, consider the test function $f_{Q_{1}}^\xi$ of   \eqref{testf} and  set $\Omega^{\tau,\xi}(Q_1) = \carl{Q_1} \setminus \cup_{R \in B^{\tau,\xi}(Q_1)} \carl{R}$ with $B^{\tau,\xi}(Q_1)$ defined in Lemma \ref{lemma7}.

\begin{lem}  \label{lem:parameters}
Suppose that $\sigma_3$, $\sigma_4$ and $\sigma_5$  are chosen as in Lemma \ref{lemma7} so that \eqref{geometric} holds.
Then there exists a choice of $\sigma_w$, $\sigma_1$ and $\sigma_2$  so that for all  $Q_{0}, Q_{1}\in \dyadic$ with $Q_{1}\subseteq Q_{0}$, and all  $\nu$ and $ \tau$,
\begin{equation} \label{43}
|\tilde{\gamma}_t(x)| \leq C|\gamma_tE_t(f^\xi_{Q_1})(x)|, \quad \mathrm{for}\,(x,t) \in \Omega^w(Q_0) \cap \Omega^{\tau,\xi}(Q_1),
\end{equation}
 where $C > 0$ depends only on the choice of $\sigma_w$, $\sigma_1$, $\sigma_2$, $\sigma_3$, $\sigma_4$ and $\sigma_5$.
\end{lem}

\begin{proof}  We assume that $\Omega^w(Q_0) \cap \Omega^{\tau,\xi}(Q_1)$  is non-empty, otherwise, there is nothing to prove. Recall that  it is a union of Whitney boxes $W_{Q'}$ and \eqref{43} follows from $\left| \frac{\tilde{\gamma}_t(x)}{|\tilde{\gamma}_t(x)|} \big( E_{Q'}(f^\xi_{Q_1})\big) \right|\gtrsim 1$ for $(x,t)\in  W_{Q'}$ with $\tilde{\gamma}_t(x)\ne 0$.
For  a Whitney box $W_{Q'} \subseteq\Omega^w(Q_0) \cap \Omega^{\tau,\xi}(Q_1)$ we have that
\begin{align*}
|(\ln w)_{Q'} - (\ln w)_{Q_0}| & \leq \sigma_w, \\
|E_{Q'}(f^\xi_{Q_1})| & \leq \frac{1}{\sigma_4} \,\,\, \mbox{and} \\
 \re  \left(S_{Q_1}^\tau(\xi),\begin{bmatrix}
        w_{Q'}/w_{Q_1} &  0  \\
        0  & I
     \end{bmatrix}E_{Q'}(f^\xi_{Q_1})\right) & \geq \sigma_5.
\end{align*}
 The last two inequalities are the definition of $\Omega^{\tau,\xi}(Q_{1})$.  The first  comes from the fact  $Q'$ is not contained in a  cube of $B^w(Q_{0})$ by Proposition \ref{prop4}. As $Q'\subseteq Q_{1}\subseteq Q_{0}$, $Q_{1}$ is also not contained in a cube of  $B^w(Q_{0})$ and    we also  have
\[ |(\ln w)_{Q_{1}} - (\ln w)_{Q_0}|  \leq \sigma_w.
\]
 Moreover, recall  that  if $\tilde{\gamma}_t(x)\ne 0$,  then
\[
\left|\frac{\tilde{\gamma}_t(x)}{|\tilde{\gamma}_t(x)|} - \nu\right| \le \sigma_{1}.
\]
Finally, if $\tilde{\gamma}_t(x)\ne 0$ and $(x,t)\in W_{Q'}$ then
\[
|\ln(w_{Q'}) - (\ln w)_{Q'} - \tau| < \sigma_2.
\]
Clearly then,   we may assume the  six  inequalities above.

We begin by observing that
\[
 \begin{bmatrix}
        w_{Q'}/w_{Q_1} &  0  \\
        0  & I
     \end{bmatrix} S_{Q_1}^\tau = \begin{bmatrix}
        e^{\ln (w_{Q'})-\tau - (\ln w)_{Q_1}} &  0  \\
        0  & I
     \end{bmatrix}
\]
and
\begin{align}
 |\ln(w_{Q'}) - \tau - (\ln w)_{Q_1}|
& \leq |\ln(w_{Q'}) - (\ln w)_{Q'} - \tau| \nonumber \\
& + |(\ln w)_{Q'} - (\ln w)_{Q_0}| + |(\ln w)_{Q_1} - (\ln w)_{Q_0}|  \nonumber\\
&  \leq \sigma_2 + \sigma_w + \sigma_w. \label{1}
\end{align}
Recall that we chose $\sigma_3$ in \eqref{sigma3} so that $0 < c\sigma_3^{\delta} \leq e^{-2c_{0}}/2  \leq 1/2$. Therefore, using Lemma \ref{lemma6},
\[
\left| \nu \left(E_{Q'}(f^\xi_{Q_1})\right) \right| \geq \left| \nu \left(\xi\right) \right| - \left| \nu \left(E_{Q'}(f^\xi_{Q_1}) - \xi\right) \right| \geq 1 - c\sigma_3^{\delta} \geq 1/2
\]
and
\[
 \re  \left( \xi, E_{Q'}(f^\xi_{Q_1})\right) = \left( \xi, \xi\right) +  \re  \left( \xi,  E_{Q'}(f^\xi_{Q_1})-\xi\right) \leq 1 + c\sigma_3^{\delta} \leq 2,
\]
so
\[
\left| \nu \left(E_{Q'}(f^\xi_{Q_1})\right) \right| \geq \frac{1}{4}  \re  \left( \xi, E_{Q'}(f^\xi_{Q_1})\right).
\]
It then follows that for $(x,t)\in W_{Q'}$ with $\tilde{\gamma}_t(x)\ne 0$,
\begin{align*}
& \left| \frac{\tilde{\gamma}_t(x)}{|\tilde{\gamma}_t(x)|} \left( E_{Q'}(f^\xi_{Q_1})\right) \right| \\
& \geq \left| \nu \left(E_{Q'}(f^\xi_{Q_1})\right) \right| - \left|\left(\frac{\tilde{\gamma}_t(x)}{|\tilde{\gamma}_t(x)|} - \nu\right)\left( E_{Q'}(f^\xi_{Q_1})\right) \right| \\
& \geq \frac{1}{4} \re   \left( \xi, E_{Q'}(f^\xi_{Q_1})\right) - \frac{\sigma_1}{\sigma_4} \\
& = \frac{1}{4}  \re  \left( \begin{bmatrix}
        e^{\ln(w_{Q'} ) -\tau - (\ln w)_{Q_1}}&  0  \\
        0  & I
     \end{bmatrix}\xi, E_{Q'}(f^\xi_{Q_1})\right) \\
& \quad + \frac{1}{4} \re  \left(\begin{bmatrix}
        1-e^{\ln(w_{Q'})-\tau - (\ln w)_{Q_1}} &  0  \\
        0  &  0
     \end{bmatrix}\xi, E_{Q'}(f^\xi_{Q_1})\right) - \frac{\sigma_1}{\sigma_4} \\
& \geq \frac{1}{4} \re  \left( S_{Q_1}^\tau(\xi), \begin{bmatrix}
        w_{Q'}/w_{Q_1} &  0  \\
        0  & I
     \end{bmatrix}E_{Q'}(f^\xi_{Q_1})\right) - \frac{e}{\sigma_4} (\sigma_2 + 2\sigma_w) - \frac{\sigma_1}{\sigma_4} \\
& \geq \frac{1}{4}\sigma_5 - \frac{e}{\sigma_4} (\sigma_2 + 2\sigma_w + \sigma_1).
\end{align*}
We have used \eqref{1} and  $|1-e^u|\le e u$ if $u$ is a real with $|u|\le 1$, assuming $\sigma_2 + 2\sigma_w\le 1$.
We have already chosen $\sigma_4$ and $\sigma_5$, but we are still free to choose $\sigma_w$, $\sigma_1$ and $\sigma_2$ small  so that \eqref{43} holds.
\end{proof}

\begin{proof}[Proof of Theorem \ref{th:main}]  By Corollary \ref{cor:cor1} and  Lemma \ref{lemma5}, we know it is enough to show, for fixed $\nu$ and $\tau$,  and   for any cube $Q_0$, that
\begin{equation} \label{indeed}
\iint_{\Omega^w(Q_0)} |\widetilde{\gamma}_t(x)|^2\,  \frac{dw(x)dt}{t} \leq Kw(Q_0).
\end{equation}
   Fix $Q_{1}\in \Delta$ with $Q_{1}\subseteq Q_{0}$.  Having fixed the parameters in Lemma \ref{lem:parameters}, we apply  \eqref{43} in the first inequality to obtain
\begin{align*}
 \iint_{\Omega^w(Q_0)\cap\Omega^{\tau,\xi}(Q_1)}& |\widetilde{\gamma}_t(x)|^2\,  \frac{dw(x)dt}{t} \\
& \leq \iint_{\carl{Q_1}} |\gamma_tE_t(f^\xi_{Q_1})(x)|^2\,  \frac{dw(x)dt}{t} \\
& \lesssim \iint_{\carl{Q_1}} |(Q^B_tf^\xi_{Q_1})(x)|^2\,  \frac{dw(x)dt}{t} \\
& \quad + \iint_{\carl{Q_1}} |(Q^B_t - \gamma_tE_t)(f^\xi_{Q_1} - \etaq\xi)(x)|^2\,  \frac{dw(x)dt}{t} \\
& \quad + \iint_{\carl{Q_1}} |(Q^B_t - \gamma_tE_t)(\etaq\xi)(x)|^2 \, \frac{dw(x)dt}{t}.
\end{align*}
Since $Q_{t}^Bf^\xi_{Q_1}= \frac t{\sigma_{3}\ell(Q_{1})}P_{t}^B (\sigma_{3}\ell(Q_{1})DBf^\xi_{Q_1} ) $, one has  that
\[
\iint_{\carl{Q_1}} |(Q^B_tf^\xi_{Q_1})(x)|^2\,  \frac{dw(x)dt}{t} \lesssim \int_{0}^{\ell(Q_{1})} \frac {t^2\|\sigma_{3}\ell(Q_{1})DBf^\xi_{Q_1}\|^2}{(\sigma_{3}\ell(Q_{1}))^2} \,  \frac {dt}t \lesssim w(Q_1),
\]
and, by Proposition \ref{lem:ppa} because $f^\xi_{Q_1} - \etaq\xi \in \ran(D)$,
\[
\iint_{\carl{Q_1}} |(Q^B_t - \gamma_tE_t)(f^\xi_{Q_1} - \etaq\xi)(x)|^2\,  \frac{dw(x)dt}{t} \lesssim \|f^\xi_{Q_1} - \etaq\xi\|^2 \lesssim w(Q_1).
\]
 For the last term, using that by definition $Q^B_t - \gamma_t E_t$ annihilates constants and $E_t((\etaq -1)\xi)(x)=0$ when $(x,t) \in \carl{Q_{1}}$, we can rewrite
\[(Q^B_t - \gamma_t E_t) (\etaq \xi)(x)= (Q^B_t - \gamma_t E_t) ((\etaq -1) \xi)(x) =  Q^B_t ((\etaq -1)\xi)(x).\]
Using off-diagonal estimates for $Q^B_t$ as in Lemma \ref{lem:gammat}, one can easily show $\iint_{\carl{Q_1}} |Q^B_t ((1_{2Q_{1}}-1)\xi)(x)|^2  \frac{dw(x)dt}{t} \lesssim w(Q_{1})$. Next, decompose $2Q_{1}\setminus Q_{1} = \pd_{a(t)} \cup (2Q_{1}\setminus \pd_{a(t)})$ where $a(t)= \sqrt {t/\ell(Q_{1})} \ell(Q_{1})$ and $\pd_{a}=\{y\notin Q_{1}\, ; \, d(y,Q_{1})\le a\}$.  Again,  using the off-diagonal estimates for each $t$, the function $1_{2Q_{1}\setminus \pd_{a(t)}}$ contributes  $w(Q_{1})$. It  remains to control  the integral corresponding to $1_{\pd_{a(t)}}$. From the uniform boundedness of $Q^B_{t}$, one has
\begin{align*}
\iint_{\carl{Q_1}} |Q^B_t (1_{\pd_{a(t)}}\xi)(x)|^2  \frac{dw(x)dt}{t}    &  \lesssim  \int_{0}^{\ell(Q_{1})} w(\pd_{a(t)})  \frac{dt}{t}  \\
    &  \lesssim  \int_{0}^{\ell(Q_{1})} \bigg(\frac{|\pd_{a(t)}|}{|Q_{1}|}\bigg)^\sigma   \frac{dt}{t} \, w(Q_{1})\\
    &   \lesssim w(Q_{1})
\end{align*}
using  \eqref{eq:ainfty} and $\frac{|\pd_{a(t)}|}{|Q_{1}|} \lesssim \big(\frac t{\ell(Q_{1})}\big)^{1/2}$ obtained from elementary observations. Summarizing the estimates above, we have proved that
$$
\iint_{\Omega^w(Q_0)\cap\Omega^{\tau,\xi}(Q_1)} |\widetilde{\gamma}_t(x)|^2\,  \frac{dw(x)dt}{t}
\lesssim w(Q_{1}).
$$

 We can now prove \eqref{indeed}.
Define
\begin{align*}
B^{\tau,\xi}_0(Q_0) & = \{Q_0\}, \quad B^{\tau,\xi}_1(Q_0) = B^{\tau,\xi}(Q_0), \\
B^{\tau,\xi}_{j+1}(Q_0) & = \cup_{R \in B^{\tau,\xi}_j(Q_0)}B^{\tau,\xi}(R) \,\,\, \mbox{for} \,\,\, j = 1,2,\dots, \\
\mbox{and} \,\,\, B^{\tau,\xi}_*(Q_0) & = \cup_{j=0}^\infty B^{\tau,\xi}_j(Q_0).
\end{align*}
Using
\[
\iint_{\Omega^w(Q_0)} |\widetilde{\gamma}_t(x)|^2\,  \frac{dw(x)dt}{t} = \sum_{Q_1 \in B^{\tau,\xi}_*(Q_0)} \iint_{\Omega^w(Q_0) \cap \Omega^{\tau,\xi}(Q_1)} |\widetilde{\gamma}_t(x)|^2 \, \frac{dw(x)dt}{t},
\]
 and summing the estimate above together with an iteration of  Lemma \ref{lemma7} imply
\begin{align*}
& \iint_{\Omega^w(Q_0)} |\widetilde{\gamma}_t(x)|^2\,   \frac{dw(x)dt}{t} \lesssim \sum_{Q_1 \in B^{\tau,\xi}_*(Q_0)} w(Q_1) \\
& \leq \sum_{j=0}^\infty (1-\sigma_6)^j w(Q_0) \lesssim w(Q_0),
\end{align*}
which proves \eqref{indeed} and with it Theorem \ref{th:main}.
\end{proof}

\subsection{The case of block matrices}
We show how to simplify the argument in this case.
Recall that $B$ is a $(n+1) \times (n+1)$ matrix. Assume here that it is block diagonal, namely
$$
B(x)= \begin{bmatrix} a(x) & 0 \\ 0 & d(x)
\end{bmatrix}
$$
with $a(x) $ scalar-valued and $d(x)$ $n\times n$ matrix-valued. Define the normal and tangential spaces
$$\mH_{\no}=\begin{bmatrix} L^2(\R^n, w; \C) \\ 0
\end{bmatrix} \quad \mathrm{and} \quad  \mH_{\ta}=\begin{bmatrix} 0\\ L^2(\R^n, w; \C^n)
\end{bmatrix}.$$
In this case, both operators $BD$ and $DB$  swap the normal and tangential spaces. So do $Q_{t}^B$ and  multiplication by $\gamma_{t}$. This means that
$$
\gamma_{t}(x)= \begin{bmatrix} 0 & \alpha_{t}(x) \\ \beta_{t}(x) & 0
\end{bmatrix}
$$
so that   the Carleson function norms
for $|\alpha_{t}(x)|^2$ and $|\beta_{t}(x)|^2$ can be estimated separately. The normal and tangential parts of our test functions can be used in two separate  stopping-time arguments,  which do not require the Corona decomposition (Proposition \ref{prop4}),   following the usual proof in the unweighted case,  since for each stopping-time we use the average against \emph{one} measure: either $dx$ or $dw$.

\subsection{A vector-valued extension}\label{sec:vv}

The proof of the main quadratic estimate carries straightforwardly to the case of systems where
\begin{itemize}
  \item $\mH= L_{2}(\R^n, w; \C^{m(1+n)})$,  where $ \C^{m(1+n)}= (\C^m)^{1+n}$,
  \item $D$ acts componentwise on $\mH$ by $(Du)^\alpha= D(u^\alpha)$ for $\alpha=1, \ldots, m$  (in other words, the new $D$ is $D\otimes I_{\C^m}$ but we shall not use this notation),  and
  \item  $B(x)$ is an $(n+1)\times (n+1)$ matrix whose entries are $m\times m $ matrices and the multiplication by $B(x)$ is assumed to be bounded on $\mH$ and accretive on $\clos{\ran(D)}$.

\end{itemize}

\subsection{Consequences}\label{sec:consequences}

We gather  here  some consequences  on the functional calculus for the convenience of the reader.

\begin{prop}\label{prop:equi} Let  $w,D$ and $B$ be as above. If $T=DB$ or $T=BD$, then one has the equivalence
\begin{equation}\label{eq:sfT}
  \int_0^\infty\|tT(1+t^2T^2)^{-1} u \|^2 \, \frac{dt}t \sim \|u\|^2, \qquad  \text{for all }\ u\in \clos{\ran(T)}.
\end{equation}
\end{prop}

\begin{proof} First, the square function estimate for $BD$  follows from that for $DB$. Indeed, on $\nul(BD)$,
$tBD(1+t^2(BD)^2)^{-1} u =0$. On $\clos{\ran(BD)}$, $BD$ is similar to $DB$ so the square function inequality  for $BD$   follows. We  conclude  using the splitting $\mH= \nul(BD) \oplus  \clos{\ran(BD)}$. Now, if one changes $B$ to  $B^*$, this means we have proved
\eqref{eq:sfBD} for both $T=DB$ (resp. $T=BD$) and its adjoint. It is classical (\cite{ADMc}) that this implies the equivalence on the range.
\end{proof}

The next result summarizes consequences of quadratic estimates that are needed.

\begin{prop}\label{prop:SFimpliesFC} Let  $T$ be an $\omega$-bisectorial operator on  a separable Hilbert space $\mH$ with $0\le \omega<\pi/2$.  Assume that the quadratic estimate
\begin{equation}
  \int_0^\infty\|tT(1+t^2T^2)^{-1} u \|^2 \, \frac{dt}t \sim \|u\|^2 \   \text{holds for all }\ u\in \clos{\ran(T)}.
\end{equation}Then, the following statements hold.
\begin{itemize}
  \item $T$ has a bounded holomorphic functional calculus on $\clos{\ran(T)}$ on any bisector $|\arg (\pm z)| <\mu$ for any $\omega<\mu<\pi/2$, which can be extended to all $\mH$ by setting $f(T)=f(0)I$ on $\nul(T)$ whenever $f$ is also defined at 0.
\item The comparison  \begin{equation} \label{eq:psiT}
  \int_0^\infty\|\psi(tT) u \|^2 \, \frac{dt}t \sim \|u\|^2 \  \text{holds for all }\ u\in \clos{\ran(T)}.
\end{equation}
for any $\omega<\mu<\pi/2$ and  for any  holomorphic function $\psi$ in the bisector $|\arg (\pm z)| <\mu$, which is not identically zero on each connected component of the bisector and which satisfies
$|\psi(z)|\le C\inf (|z|^\alpha, |z|^{-\alpha})$ for some $C<\infty$ and $\alpha>0$.
  \item The operator $\sgn(T)$ is a bounded involution on $\clos{\ran(T)}$.
  \item  $\clos{\ran(T)}$ splits topologically into two spectral subspaces
\begin{equation}     \label{eq:hardysplit}\clos{\ran(T)}=\mH^+_{T}\oplus \mH^-_{T}
\end{equation}
  with $\mH^\pm_{T}=E_{T}^\pm(\clos{\ran(T)})$ and $E_{T}^+=\chi^\pm(T)$ are projections with $\chi^\pm(z)=1$ if $\pm \re z>0$ and $\chi^\pm(z)=0$ otherwise.
\item The operator $|T|=sgn(T)T = \sqrt {T^2}$ with $\dom(|T|)=\dom(T)$ is a $\omega$-sectorial operator and
$-|T|$ generates an analytic semigroup of operators $(e^{-z|T|})_{|\arg z| <\pi/2 - \omega}$.
\item For $h\in \dom(T)$, $h\in \mH^\pm_{T}$ if and only if $|T|h=\pm Th.$ As a consequence
$e^{\mp zT}$ are well-defined operators on $\mH^\pm_{T}$ respectively,  and $e^{-zT}E_{T}^+$ and $e^{+zT}E_{T}^-$ are well-defined operators on $\mH$   for $|\arg z | <\pi/2 - \omega$.
\end{itemize}\end{prop}

As announced in Section \ref{sec:main}, we recall here why this implies the Kato conjecture for block diagonal
$$
B= \begin{bmatrix} a & 0 \\ 0 & d
\end{bmatrix}
$$
identifying the functions $a$ and $d$ with the corresponding multiplication operators.
We have
$$BD= \begin{bmatrix}  0 & a \divv_{w}  \\ - d\nabla  & 0
\end{bmatrix} \ , \ (BD)^2= \begin{bmatrix} -a\divv_{w} d\nabla & 0 \\ 0 & - d\nabla a  \divv_{w}
\end{bmatrix}, $$ so that for $u\in H^1(\R^n, w;\C^m)$, $v=\begin{bmatrix}
      u    \\
      0
\end{bmatrix} \in \dom(BD)= \dom  (|BD|)$ and
$$\|\sqrt {-a\divv_{w} d\nabla} u\| \sim \| |BD|v \| \sim \|BD v\|  \sim \|d\nabla u \| \sim \|\nabla u\|.
$$

\section{Representations for solutions of degenerate elliptic systems}\label{sec:rep}

From now on, we write  points in the upper half-space  $\R^{1+n}_+$ as $\bx=(t,x)$, $t>0, x\in \R^n$.

\subsection{From second order to first order}

We shall now follow closely \cite{AA1}, and its extension \cite{R}, but in the weighted setting. It is necessary to have these references handy.  The estimates  of these two articles obtained in  abstract Hilbert spaces  evidently apply here.  Some other estimates use harmonic analysis (tent spaces, maximal functions). Thus we shall try to extract the relevant information and give proofs only when the argument uses a particular feature of the weighted situation.

We recall the notation $\mH=L^{2}(\R^n,w;\C^{m(1+n)})$ and  use $\mH^{0}=\clos{\ran(D)}$ where $D$ was defined in Section  \ref{sec:vv}. Beware that in \cite{AA1}, $\mH$ was taken as $\clos{\ran(D)}$.
We continue to use $\|\ \|$ to denote the norm in $\mH$, and occasionally use other notation when needed.

We construct solutions $u$ to the divergence form system (\ref{eq:divform}),
by  solving the equivalent  vector-valued ODE (\ref{eq:firstorderODE}) below for the $w$-normalized conormal gradient
$$
  f=\nabla_{w^{-1}A} u= \begin{bmatrix} \pd_{\nu_{w^{-1}A}}u \\ \nabla_x u \end{bmatrix},$$
and
$\pd_{\nu_{w^{-1}A}}u$ denotes the upward  (hence inward for $\reu$) $w$-normalized conormal derivative of $u$.

Using the normal/tangential  decomposition for  $\C^{m(1+n)}= \C^m \oplus \C^{mn}= \C^m \oplus (\C^m \otimes \C^n)$ (see Section \ref{sec:main}),  we  write matrices acting on $\C^{m(1+n)}$   as
$$
  M= \begin{bmatrix} M_{\no\no} & M_{\no\ta} \\ M_{\ta\no} & M_{\ta\ta} \end{bmatrix},
$$
the entries being matrices acting from and into the various spaces in the splitting.

\begin{prop}\label{prop:hat}
  The transformation
$$
   C\mapsto \widehat C:=   \begin{bmatrix} I & 0  \\
     C_{\ta\no} & C_{\ta\ta} \end{bmatrix} \begin{bmatrix} C_{\no\no} &  C_{\no\ta}  \\
     0 & I \end{bmatrix}^{-1} =  \begin{bmatrix} C_{\no\no}^{-1} & -C_{\no\no}^{-1} C_{\no\ta}  \\
     C_{\ta\no}C_{\no\no}^{-1} & C_{\ta\ta}-C_{\ta\no}C_{\no\no}^{-1}C_{\no\ta} \end{bmatrix}
$$
is a self-inverse bijective transformation of the set of operator-valued matrices which are bounded on $\mH$ and   accretive on $\clos{\ran(D)}$.
\end{prop}

The proof is analogous to that of \cite{AA1}.

  We set
  $$
  \widehat {w^{-1}A}= B$$ in what follows.
  Our assumption is that as a multiplication operator,  $w^{-1}A(t,\cdot)$ is bounded on $\mH$  and  accretive on $\clos{\ran(D)}$ for a.e.~$t>0$ with uniform bounds with respect to $t$. In particular,  the matrix $w^{-1}A_{\no\no}(t,\cdot)$ is invertible as an operator acting on $L^2(\R^n; w; \C^m)$, hence it is also invertible in $L^\infty(\R^n, \mL(\C^{m}))$, with uniform bounds a.e.~in $t>0$.  Thus, $B(t,\cdot)$ is also a multiplication operator.

We now introduce some notation.
Let
$$
    \mD_{w}=\begin{bmatrix} C_{0}^\infty(\R^{1+n}_{+}; \C^m)  \\
    w^{-1} C_{0}^\infty(\R^{1+n}_{+}; \C^{mn})\end{bmatrix}.
    $$

    Let $\mC url_{\ta,0}= \{f \in \mD'(\R^n; \C^{m(1+n)})\, ;\, \curl_{x} f_{\ta}=0\}$, where the curl operator is computed componentwise.
    Let $\mH_\loc=L^2_\loc(\R^n, w;\C^{m(1+n)})$.

\begin{prop}  \label{prop:divformasODE}
For a pair of coefficient matrices $A$ and $B$ related by  $A= w\widehat B$, or equivalently $B= \widehat {w^{-1}A}$,
the pointwise map $g\mapsto f= \begin{bmatrix}  (w^{-1} A g)_\no \\ g_\ta \end{bmatrix} $ gives
a one-to-one correspondence, with inverse $g= \begin{bmatrix}  (B f)_\no \\ f_\ta \end{bmatrix} $,
between solutions $g$ to the equations
\begin{equation}  \label{eq:firstorderdiv}
     \begin{cases}
     g\in L^2_\loc(\R_+;\mH_\loc), \\
     \divv_{t,x} (Ag)=0, \\
     \curl_{t,x} g=0,
     \end{cases}
\end{equation}  in the sense of distributions on $\R^{1+n}_{+}$
and solutions $f$  to the generalized Cauchy--Riemann equations
\begin{equation}  \label{eq:firstorderODE}
  \begin{cases}
  f\in L^2_\loc(\R_+;\mH_\loc\cap \mC url_{\ta,0}), \\
  \pd_t f+ DB f=0,
   \end{cases}
\end{equation}
 in the weak sense
\begin{equation}
\label{firstorderweak}
\int_{0}^\infty -(f, \pd_{t}\varphi) + (Bf,D\varphi)\, dt =0 \quad \forall \varphi\in \mD_{w},
\end{equation}
where $(\ ,\ )$ is the complex inner product with respect to $dw$.  \end{prop}

The proof is almost completely identical to the one in \cite{AA1}.

\begin{proof}
The transformation $g\mapsto f= \begin{bmatrix}  (w^{-1} A g)_\no \\ g_\ta \end{bmatrix} $ is easily seen to be invertible on $L^2_\loc(\R_+;\mH^\loc)$.  Consider a pair of functions $g$ and $f$ in $L^2_\loc(\R_+;\mH^\loc)$ related in this fashion.
  Equations (\ref{eq:firstorderdiv}) for $g$ are equivalent to
\begin{equation}
\begin{cases}
  \pd_t (Ag)_\no + \divv_x( A_{\ta\no} g_\no + A_{\ta\ta}g_\ta) =0, \\
  \pd_t g_\ta - \nabla_x g_\no =0, \\
  \curl_x g_\ta =0,
\end{cases}
\end{equation}
each in the sense of distributions on $\R^{1+n}_{+}$.
  The last equation is equivalent to $f_t=f(t,\cdot)\in  \mC url_{\ta,0}$.
  Moreover, using that $(w^{-1}Ag)_\no= f_\no$, $g_\ta= f_\ta$ and
  $g_\no= (Bf)_\no= A_{\no\no}^{-1}(wf_\no- A_{\no\ta}f_\ta)$,
  the first two equations are seen to be equivalent to the equation $\pd_t f+ DB f=0$ in the prescribed sense.
  \end{proof}

  Next, the strategy in \cite{AA1} is to integrate the weak differential equation  \eqref{firstorderweak}
  to obtain an equivalent formulation in the Duhamel sense. Again,  this  can be followed almost line by line, once we have the following density lemma.

  \begin{lem} The space $\mD_{w}$ is dense in  $H^{1}_\comp(\R_+;\mH) \cap L^2_\comp(\R_+;\dom(D))$, where the subscript $\comp$ means that elements have compact support in $\R_{+}$. Thus, if $f\in L^2_\loc(\R_+;\mH^{0})$, \eqref{firstorderweak} holds for any $\varphi\in H^{1}_\comp(\R_+;\mH) \cap L^2_\comp(\R_+;\dom(D))$ if it does for any $\varphi\in\mD_{w}$.
\end{lem}

 \begin{proof} The density of $\mD_{w}$ in  $H^{1}_\comp(\R_+;\mH) \cap L^2_\comp(\R_+;\dom(D))$ can be easily established using (2) in Proposition \ref{prop:D} and standard truncation and regularization in the $t$-variable.  If  $f\in L^2_\loc(\R_+;\mH^{0})$ and $\varphi\in H^{1}_\comp(\R_+;\mH) \cap L^2_\comp(\R_+;\dom(D))$ then the integral  in \eqref{firstorderweak} makes sense and vanishes by approximating $\varphi$ by elements in $\mD_{w}$.
\end{proof}

In the above proposition, we are mostly interested in  having $g\in L^2_\loc(\R_+;\mH)$. Recall that
    $$\mH^{0}=\clos{\ran(D)}= L^{2}(\R^n,w;  \C^{m(1+n)})\cap \mC url_{\ta,0}=\mH\cap \mC url_{\ta,0}.$$
    In particular, $g\in L^2_\loc(\R_+;\mH)$ if and only if $f\in L^2_\loc(\R_+;\mH^{0})$ and we can apply the above lemma.
Formally writing $\pd_{t}f+ DB_{0}f= D(\E f)$, $\E=B_{0}-B$, where $B_{0}$ is now  multiplication by a $t$-independent matrix,
the integration of \eqref{firstorderweak} leads to the following equation
\begin{equation}   \label{eq:inteqroadmap}
  f_{t}= e^{-tDB_{0}} E_{0}^+ h^+ + (S_Af)_{t},
\end{equation}
for a unique $h^+\in \mH^+_{DB_{0}}$
and where $S_A$ is the vector-valued singular integral operator given
\begin{equation}   \label{eq:firstformalSAdefn}
    (S_A f)_t :=  \int_0^t e^{-(t-s)DB_0} E_0^+ D \E_s f_s \, ds - \int_t^\infty e^{(s-t)DB_0} E_0^- D\E_s f_s\,  ds.
\end{equation}
Here $E_{0}^\pm=\chi^\pm(DB_{0})$ are the projections defined in Section \ref{sec:consequences}
and $ \mH^+_{DB_{0}}:= E_{0}^\pm\mH $ are the ranges of  the respective projections. We also use  the notation $g_{t}=g(t,\cdot)$.
This operator can be rigorously defined using  the maximal regularity operator for $|DB_0|$ viewed from the operational calculus point of view as $F(|DB_0|)$ with $F(z)$ being the operator-valued analytic function given by
$$
(F(z)g)_t:= \int_0^t ze^{-(t-s) z}g_s\,  ds,  \ \re z>0,
$$
so  \eqref{eq:firstformalSAdefn} becomes
\begin{equation}    \label{eq:inteqopcalcroadmap}
   S_A= F(|DB_0|) \hE_0^+\E + F^*(|DB_0|) \hE_0^- \E,
\end{equation}
where $\hE_0^\pm$ are bounded operators on $\mH$ such that $E_0^\pm D= (DB_0)\hE_0^\pm$. These two representations and Proposition \ref{prop:equi} allow us to prove most relevant boundedness results concerning the regularity and Neumann problems. For the Dirichlet problem, they can be used as well  in an appropriate sense (see \cite{R}) but there is  another useful representation using the operator \begin{equation}    \label{eq:tildeSA}
  (\tS_A f)_t:=  \int_0^t e^{-(t-s)B_0D} \tE_0^+ \E_s f_s \, ds - \int_t^\infty e^{(s-t)B_0D} \tE_0^- \E_s f_s \, ds,
\end{equation} where $\tE_0^\pm=\chi^\pm(B_{0}D)$, and the vector field defined by
$$
   v_{t}  := e^{-tB_{0}D} \tE_{0}^+  \tilde h^+ + (\tS_A f)_{t},
   $$
   for some $\tilde h^+$. From the  intertwining property $b(DB_0)D= Db(B_0D)$ of the functional
calculi of $DB_0$ and $B_0D$, one has $D\tS_A =S_{A}$, so that the relation $D\tilde h^+=h^+$, which uniquely determines a choice $\tilde h^+\in \mH^+_{B_{0}D}$, shows  that $Dv=f$. Solutions $u$ to the second order equations are related to $v$ in the sense that   there exists a constant $c\in \C^m$ such that
   $
   u =c  -v_\no.
$
This means that the tangential part $-v_{\ta}$ encodes a conjugate to the solution $u$. This notion of conjugate was further developed in \cite{AA2}.

These representations are justified provided one has  the operator bounds in the next section.

\subsection{Functions spaces and operator estimates}

Here, we give  the definition of the functions spaces associated to the BVPs. What changes compared to \cite{AA1} is that
the Lebesgue measure $dx$ on $\R^n$ is replaced by $dw$ and Lebesgue measure $d\bx=dtdx$ on $\R^{1+n}$  by $d\uw=dtdw$ where $\uw$ is the $A_{2}$ weight on $\R^{1+n}$ defined by $\uw(t,x)=w(x)$. The only property required for $w$  in this section is the doubling property of $dw$, except when we use the quadratic estimate which uses the $A_{2}$ property.  Also we incorporate the posterior duality and multiplier results of
 \cite{HR} (for $dx$), which can be extended to the weighted setting too. See below.

\begin{defn}    \label{defn:NTandC} For an $L^q_{loc}$ function $f$, $1\le q\le \infty$, define
$W_{q}f(t,x)= \left(\barint_{\hspace{-2pt}W(t,x)} |f|^q \, d\uw\right)^{1/q}$ with the usual essential supremum definition if $q=\infty$ and where $W(t,x):= (c_0^{-1}t,c_0t)\times B(x;c_1t)$ is a Whitney region, for some fixed constants $c_0>1$, $c_1>0$.
The {\em weighted non-tangential maximal function} of an $L^2_{\loc}$ function $f$ in $\R^{1+n}_+$ is
$$
  \tN(f)(x):= \sup_{t>0}  W_{2}f(t,x), \qquad x\in \R^n,
$$
The {\em weighted Carleson functional} of an $L^1_{\loc}$ function $f$ is
$$
  Cf(x) := \sup_{Q\ni x} \frac 1{w(Q)} \int_{(0, \ell(Q))\times Q} |f(t,y)| \, d\uw(t,y),\qquad x\in\R^n,
$$
where the supremum is taken over all cubes $Q$ in $\R^n$ containing $x$, with $\ell(Q)$ denoting their sidelengths.
The {\em modified weighted Carleson norm} of a measurable function $g$ in $\R^{1+n}_+$ is
$$
  \|g\|_* := \| C(W_{\infty}(|g|^2/t))\|_\infty^{1/2}.
$$
\end{defn}

We will use the modified Carleson norm to measure the size of perturbations
of $t$-independent coefficients $A_0$.
  The proof of Lemma 2.2 in \cite{AA1} adapts to show that if there exists $A_{0}(x)$ with
  $\|w^{-1}(A-A_0)\|_* < \infty$, then it is unique and $w^{-1}A_{0}$ is bounded, and accretive on $\clos{\ran(D)}$, so that we may call $A_{0}$ the trace of $A$.

  \begin{defn}    \label{defn:XY}
   Define the Banach/Hilbert spaces
\begin{align*}
  \mX & := \sett{ f\in L^2_{loc}(\R^{1+n}_+; \C^{m(1+n)})}{ \|\tN(f)\|<\infty }, \\
  \mC & :=   \sett{f\in L^2_{loc}(\R^{1+n}_+; \C^{m(1+n)}) }{ \|C(W_{2}f)\|<\infty },\\
  \mY &:= \sett{f\in L^2_{loc}(\R^{1+n}_+; \C^{m(1+n)})}{\int_0^\infty \| f_t \|^2 \, tdt < \infty}, \\
  \mY^* &:= \sett{f\in L^2_{loc}(\R^{1+n}_+; \C^{m(1+n)}) }{\int_0^\infty \| f_t \|^2\,  \frac{dt}t < \infty},
\end{align*}
with the obvious norms.
\end{defn}

We use the same notation as in \cite{AA1}, but of course, here all norms are weighted. Note that $\mY^*$ is the dual space of
$\mY$ with respect to  the inner product $\langle\ , \ \rangle_{\uw}$ of $\mH=L^2(\R^{1+n}_+, d\uw ;\C^{m(1+n)})$.

\begin{lem}      \label{lem:XlocL2}
  There are estimates
$$
  \sup_{t>0} \frac 1t\int_t^{2t} \| f_s \|^2\,  ds \lesssim \| \tN(f) \|^2 \lesssim \int_0^\infty \| f_s \|^2\, \frac {ds}s.
$$
In particular $\mY^*\subseteq \mX$.
\end{lem}

 A fundamental quantity is the norm  of multiplication operators mapping $\mX$ into $\mY^*$ or $\mY$ into $\mC$.

\begin{lem}   \label{lem:Carleson} The dual of $\mX$  with respect to the pairing   $\langle\ , \ \rangle_{\uw}$ is $\mC$.
  For  functions $\E: \R^{1+n}_+\to \mL(\C^{m(1+n)})$, we have estimates
$$
  \|\E\|_\infty \lesssim  \|\E\|_* \sim \sup_{\|f\|_\mX=1}\|\E f\|_{\mY^*} \sim \sup_{\|f\|_\mY=1}\|\E f\|_{\mC}.
$$
\end{lem}

\begin{proof} When $w=1$, the duality was established in \cite{HR} and recently  another more direct proof  was given in \cite{Huang}. This second proof passes to the  doubling measure setting (personal communication  of Amenta and Huang). Next,
the first inequality is proved in a similar way than in  \cite{AA1}.  The equivalences for the pointwise multiplier operator norms were also established in \cite{HR}, and reproved in \cite{Huang} when $w=1$, and the latter proof extends in a doubling measure context as well  (personal communication  of Amenta and Huang).
  \end{proof}

  \begin{prop}\label{prop:X}
  Let $u\in W^{1,2}_{\loc}(\R^{1+n}_+,w) $  be such that $\|\tN(\nabla_{t,x}u)\|<\infty$. Then there exists $u_{0}\in \dot H^1(\R^n,w)$ (as defined in the proof of  Lemma \ref{lem:gradient}) such that $\|u_{t}-u_{0}\| \lesssim t$,  $\|\nabla_{x}u_{0}\| \lesssim \|\tN(\nabla_{t,x}u)\|$,  and for $dw$ almost every $x_{0} \in \R^n$,
\begin{equation}
\label{eq:pointwiset}
 \barint_{\hspace{-6pt}W(t,x_{0})} |u-u_{0}(x_{0})|^2\, d\uw \leq t^2g(x_{0}).
\end{equation}
with $g\in L^1(\R^n,w)$.  Conversely, if $u_{0}\in \dot H^1(\R^n,w)$, then $u=e^{t^2\Delta_{w}}u_{0}$ satisfies $\|\tN(\nabla_{t,x}u)\| \lesssim \|\nabla_{x} u_{0}\|$.
 \end{prop}

\begin{proof} The first part is the weighted version of a result in \cite{KP}.
First, it is easy to show   $\|u_{t}-u_{t'}\| \lesssim |t-t'|$ by using $u_{t}-u_{t'}= \int_{t'}^t \pd_{s}u_{s}\, ds$, Cauchy-Schwarz inequality and   the left hand  inequality in Lemma \ref{lem:XlocL2}. This gives the existence of  $u_{0}\in L^2_{\loc}(\R^n,w)$ with $\|u_{t}-u_{0}\| \lesssim t$  (observe that only the difference is in $L^2(\R^n,w)$).
Next, the Poincar\'e inequality and a telescopic  sum argument implies
$$
\bigg |\barint_{\hspace{-6pt}W(t,x_{0})} u \,  d\uw - \barint_{\hspace{-6pt}W(t',x_{0})} u \,  d\uw\bigg| \le C\tau  \tN^1(\nabla_{t,x}u)(x_{0})
$$
whenever $t,t'\le \tau$ up to using a non-tangential maximal function with appropriately large Whitney regions and $\tN^1$ is the analogue of $\tN$ with $L^1$-averages. Thus for  every $x_{0}$ where  $u_{0}(x_{0})$ exists
$$
\bigg|\barint_{\hspace{-6pt}W(t,x_{0})} u \,  d\uw  -u_{0}(x_{0})\bigg| \lesssim t \tN^1(\nabla_{t,x}u)(x_{0})
$$
and $\tN^1(\nabla_{t,x}u)\in L^2(\R^n,w)$ by equivalences of norms if we change the Whitney regions.  By the Poincar\'e inequality again,
$$
\barint_{\hspace{-6pt}W(t,x_{0})} \bigg|u-\barint_{\hspace{-6pt}W(t,x_{0})} u\bigg|^2\, d\uw \le Ct^2\tN(\nabla_{t,x}u)(x_{0})
$$
and we deduce  \eqref{eq:pointwiset} on combining the last two inequalities.
 Finally, note that  if $x_{0}, y_{0}$ are different points and $t =  10 (c_{0}+c_{1}) |x_{0}-y_{0}|$, then
$$
\bigg |\barint_{\hspace{-6pt}W(t,x_{0})} u \,  d\uw - \barint_{\hspace{-6pt}W(t',y_{0})} u \,  d\uw\bigg| \le C |x_{0}-y_{0}|  \tN^1(\nabla_{t,x}u)(x_{0})
$$
again with  slightly larger Whitney regions in the definition of $\tN^1$, so that combining with
  inequalities above, we obtain
$$
|u_{0}(x_{0})-u_{0}(y_{0})| \le C |x_{0}-y_{0}|(\tN^1(\nabla_{t,x}u)(x_{0})+ \tN^1(\nabla_{t,x}u)(y_{0})).
$$
Using the theory of Sobolev spaces on the complete doubling  metric-measure space $(\R^n, |\ |, w)$, it follows that $u_{0}\in \dot H^1(\R^n,w)$ (identified with the Haj\l ash space), see \cite{HajlashKoskela}.

The converse will be proved after Theorem \ref{thm:NTmaxandaeCV}.
\end{proof}

At this stage, we do not know if $\nabla_{t,x}u$ has  almost everywhere limits or even strong $L^2(w)$ limits in the above averaged sense (although weak $L^2(w)$ convergence can be shown as in \cite{KP}). This will be the case, however, when $u$ is a solution of our systems.
We remark that in comparison, the space defined by $\int_0^\infty \|\nabla_{t,x}u_t\|^2\, tdt<\infty$ does not have a trace on $\R^n$.

With the above notation, we can state our main theorem for  $t$-independent $B_{0}$, thus for semigroups only.

\begin{thm}\label{thm:NTmaxandaeCV} Let $T=DB_{0}$ or $B_{0}D$.
Then one has the estimate
\begin{equation}
\label{eq:Ntmax}
\|e^{-t|T|}h\|_{\mX}  \sim \|h\| \sim \|\pd_{t}e^{-t|T|}h\|_{\mY} , \ \forall h\in \clos{\ran(T)}.
\end{equation}
Furthermore, for any $h\in \mH$ (not just $\clos{\ran(T)}$), we have that the Whitney averages of $e^{-t|T|}h$ converge to $h$ in $L^2$ sense, that is for $dw$ almost every $x_{0}\in \R^n$,
\begin{equation}
\label{eq:CVae}
\lim_{t\to 0}\ \barint_{\hspace{-6pt}W(t,x_{0})} |e^{-s|T|}h-h(x_{0})|^2\, d\uw=0.
\end{equation}
In particular, this implies the $dw$ almost everywhere convergence of Whitney averages
\begin{equation}
\label{eq:CVaew}
\lim_{t\to 0}\ \barint_{\hspace{-6pt}W(t,x_{0})} e^{-s|T|}h\,  d\uw=h(x_{0}).
\end{equation}
More generally, one can replace $e^{-s|T|}$ by any $\varphi(sT)$ where $\varphi$ is holomophic and bounded in some bisector containing $\sigma(T)$ and satisfies $|\varphi(z)| \lesssim |z|^{-\alpha}$ and $|\varphi(z)-a| \lesssim  |z|^\alpha$ for some $\alpha>0$, $a\in \C$. In this case,  convergence is towards $ah$, and only the upper bound $\|\varphi({tT})h\|_{\mX}  \lesssim  \|h\| $ holds if $a=0$.
\end{thm}

The proof  of this theorem will be given in Section \ref{sec:NTmax}.
The last equivalence in \eqref{eq:Ntmax}  is nothing but \eqref{eq:psiT}, we put it here for completeness.

If $B_{0}=I$, then $T^2=D^2= \begin{bmatrix} -\Delta_{w} & 0 \\ 0 & -\nabla \divv_{w} \end{bmatrix}$, so that $$\nabla_{x} e^{t^2\Delta _{w}}u_{0}= -\bigg(De^{-t^2D^2}\begin{bmatrix} u_{0}  \\ 0  \end{bmatrix}\bigg)_{\ta}= \bigg(e^{-t^2D^2}\begin{bmatrix} 0  \\ \nabla_{x }u_{0}  \end{bmatrix}\bigg)_{\ta}
$$
and
$$\nabla_{t} e^{t^2\Delta _{w}}u_{0}= \bigg(2tDe^{-t^2D^2}\begin{bmatrix} 0  \\ \nabla_{x }u_{0}  \end{bmatrix}\bigg)_{\no}
$$
Thus, $\|\tN(\nabla_{t,x} e^{t^2\Delta _{w}}u_{0})\| \lesssim \|\nabla_{x }u_{0}\|$ follows from this result, proving the converse statement in Proposition \ref{prop:X}.

We observe that  only the weak type bound  $\|\tN(e^{-t|T|}h) \|_{L^{2,\infty}(w)} \lesssim \|h\|$ holds if $h\in \nul(T)$. Concerning  the convergence \eqref{eq:CVae}, this is new even when $w=1$    for $T=DB_{0}$ in this generality.
What was proved in \cite{AA2} is \eqref{eq:CVae} for $|B_{0}e^{-t|DB_{0}|}h-(B_{0}h)(x_{0})|^2$  (which is also true in this situation), and the removal of $B_{0}$ was done only when $B_{0}^{-1}$ is given by pointwise multiplication. It turns out this is not necessary.   This will yield the almost everywhere limits in full generality in Theorem \ref{apriori_HSNeumann} as compared to \cite{AA2}.

\begin{rem}\label{rem} the almost everywhere limit  \eqref{eq:CVaew} is stated with respect to $d\uw$, which is natural. However, as they are derived from the weighted $L^2(\uw)$ limits \eqref{eq:CVae}, using that $\uw\in A_{2}$, we also have unweighted $L^1$ averages that converge to 0 almost everywhere. This means that \eqref{eq:CVaew} holds also with Lebesgue measure replacing $d\uw$.
\end{rem}

 The next two theorems are for the $t$-dependent $S_{A}$ and $\tS_{A}$. Note that we may rewrite $S_{A}f=S(\E f):=S_{\E}f$ and $\tS_{A}f= \tS(\E f):=\tS_{\E}f$ where $\E$ may not be related to $A$. We use this notation in what follows.

\begin{thm}\label{thm:estSA} Assume that $ \|\E\|_*<\infty$. Then we have the following estimates for arbitrary $f\in \mX$.
\begin{equation}
\|S_{\E}f\|_{\mX}\lesssim \|\E\|_*\|f\|_{\mX}.
\end{equation}
The function $h^-:= -\int_0^\infty  e^{sDB_{0}}E_0^-D \E_s f_s ds$ belongs to $ E_0^-\mH=\mH^-_{DB_{0}}$ and
\begin{equation}
\|h^-\|\lesssim \|\E\|_*\|f\|_{\mX},
\end{equation}
\begin{equation}
\|S_{\E}f-e^{tDB_{0}} E_0^- h^-\|_{\mY^*}\lesssim \|\E\|_*\|f\|_{\mX},
\end{equation}
\begin{equation}  \label{eq:SAavlim}
\lim_{t\to 0} t^{-1} \int_t^{2t} \| (S_{\E}f)_{s} -h^- \|^2 \,ds =0=
\lim_{t\to \infty} t^{-1} \int_t^{2t} \| (S_{\E}f)_s \|^2 ds,
\end{equation}
and
\begin{equation}
\label{eq:CVaeSA}
\lim_{t\to 0}\ \barint_{\hspace{-6pt}W(t,x_{0})} |S_{\E}f-h^-(x_{0})|^2\, d\uw=0, \ \mathrm{for \ a.e.} \ x_{0}\in \R^n.
\end{equation}
Moreover, $\tilde h^-:= -\int_0^\infty  e^{sB_{0}D}\tE_0^- \E_s f_s \, ds$ satisfies $D\tilde h^-=h^-\in E_0^-\mH$,
\begin{equation}  \label{eq:SAavlimint}
 \| (\tS_{\E}f)_{t} -\tilde h^- \|  \lesssim t \|\E\|_*\|f\|_{\mX}.
\end{equation}
In addition, if $\|\E\|_*$ is sufficiently small
 and  $\E$ satisfies the $t$-regularity condition
$
  \|t\pd_{t}\E\|_{*}<\infty,
$
  then
  \begin{equation}
\label{ }
 \|\pd_t (S_{\E} f)\|_\mY\lesssim (\|\E\|_* + \|t\pd_t \E\|_*) \|f\|_\mX + \|\E\|_\infty \|\pd_t f\|_{\mY},
\end{equation}
and one has $t\mapsto (S_{\E}f)_{t}$ is continuous  into  $\mH$ if $\|f\|_\mX + \|\pd_t f\|_{\mY}<\infty$ with improved limits
\begin{equation}  \label{eq:SAavlimimp}
\lim_{t\to 0}  \| (S_{\E}f)_{t} -h^- \|  =0=
\lim_{t\to \infty}  \| (S_{\E}f)_t \| .
\end{equation}
 \end{thm}

\begin{thm}\label{thm:esttSA} Assume that $ \|\E\|_*<\infty$. Then we have the following estimates for arbitrary $f\in \mY$.
\begin{equation}
\|S_{\E}f\|_{\mY}\lesssim \|\E\|_*\|f\|_{\mY}.
\end{equation}
The operator $\tS_{\E}$ maps $\mY$ into $C([0,\infty); \mH)$ with
\begin{equation}
\label{ }
\sup_{t\ge 0}\|(\tS_{\E}f)_{t}\| \lesssim   \|\E\|_*\|f\|_{\mY}.
\end{equation}
Moreover, $\tilde h^-:= -\int_0^\infty  e^{sB_{0}D}\tE_0^- \E_s f_s\,  ds \in \tE_0^-\mH= \mH^-_{B_{0}D}$,
\begin{equation}  \label{eq:tSAavlimint}
 \lim_{t\to 0}\| (\tS_{\E}f)_{t} -\tilde h^- \|=0 =\lim_{t\to\infty}  \| (\tS_{\E}f)_{t} \|.
\end{equation}
Furthermore, if $p<2$ and $\tN^p$ is the $p$-modified version of $\tN$ by taking $W_{p}$ functionals on Whitney regions,
\begin{equation}
\label{ }
 \|\tN^p(\tS_{\E}f)\| \lesssim  \|\E\|_*\|f\|_{\mY}
 \end{equation}
and
\begin{equation}
\label{eq:CVaetSAint}
 \barint_{\hspace{-6pt}W(t,x_{0})} |\tS_{\E}f-\tilde h^-(x_{0})|^p\, d\uw =0 , \ \mathrm{for \ a.e.} \ x_{0}\in \R^n.
\end{equation}

\end{thm}

These two theorems can be proved following the corresponding results in  Sections 6,~7,~8,~9 and 10 of \cite{AA1} (some of the arguments were simplified in \cite{R})  and, concerning the almost everywhere convergence limits \eqref{eq:CVaeSA} and  \eqref{eq:CVaetSAint},  in Section 15 of \cite{AA2} : they  hold in  a doubling weighted context when coming to use maximal functions and Carleson estimates.  The inequality \eqref{eq:SAavlimint} is not proved in \cite{AA1} and merely sketched in \cite[Section 13]{AA2}, but is easy to prove following the same decompositions as there. We shall not give details.  We just mention that $\tilde h^-$ in Theorem \ref{thm:estSA} is not an element of $\mH$: it is only defined as a limit of the integrals  truncated away from 0 and $\infty$ in the sense that $-\int_\varepsilon^R  De^{sB_{0}D}\tE_0^- \E_s f_s \, ds$ converges to $h^-$ in $\mH$. Thus, only the difference $(\tS_{\E}f)_{t} -\tilde h^- $ in  \eqref{eq:SAavlimint} makes sense in $\mH
 $. The scalar part of $\tilde h^-$ belongs to the homogeneous Sobolev space $\dot H^1(\R^n,w;\C^m)$ (as defined in the proof of Lemma \ref{lem:gradient}) and as such is also an $L^2_{loc}(w)$ function.

\subsection{{\em A priori} estimates}\label{sec:apriori}

In this subsection, we derive a priori estimates for solutions of $\divv A \nabla u=0$ with $\nabla u \in \mX$ or $\mY$.  Again, these are obtained as in \cite{AA1}, together with \cite{AA2} for the almost everywhere statements and the improvements noticed in Theorem \ref{thm:NTmaxandaeCV}.

\begin{thm}   \label{apriori_HSNeumann}
   Consider coefficients $w^{-1}A\in L^\infty(\R^{1+n}_+; \mL(\C^{m(1+n)}))$ such that $w^{-1}A$ is  accretive on $\mH^{0}$ and assume there exists $t$-independent measurable coefficients $A_0$ such that
 $\|w^{-1} (A-A_0) \|_* <\infty$ or equivalently that $ \|\E\|_*\sim \| w^{-1}(A-A_{0})\|_{*}<\infty$ where $\E=B_{0}-B$ and $B=\widehat {w^{-1}A}, B_{0}=\widehat{w^{-1}A_{0}}$.

  Let  $u$ 
    be a weak solution of
   $\divv A\nabla u=0$ in $\reu$ with  $\|\tN(\nabla_{t,x}u)\| <\infty$.
  Then
$$
  \lim_{t\to 0} t^{-1} \int_t^{2t} \| \nabla_{s,x} u_s - g_0 \|^2 ds =0=
   \lim_{t\to \infty} t^{-1} \int_t^{2t} \| \nabla_{s,x} u_s \|^2 ds,
$$
for some $g_0 \in L^2(\R^n ,w;\C^{m(1+n)})$, with estimate
$\|g_0\|\lesssim\|\tN(\nabla_{t,x}u)\|$, which we call the gradient of $u$ at the boundary and we set
$\nabla_{t,x}u|_{t=0}:=g_{0}$.
Furthermore, one has that for $dw$ almost every $x_{0}\in \R^n$,
\begin{equation}
\label{eq:CVaewgradsol}
\lim_{t\to 0}\ \barint_{\hspace{-6pt}W(t,x_{0})} \nabla_{s,x} u  \, d\uw=g_{0}(x_{0}).
\end{equation}
All three limits hold with $\nabla u, g_{0}$ replaced by  the $w$-normalized conormal gradient $f= \nabla_{w^{-1}A}u$ and $f_{0}=\begin{bmatrix}
  (w^{-1}A_{0}g_{0})_{\perp}        \\
     ( g_{0})_{\ta}
\end{bmatrix}:=\nabla_{w^{-1}A}u|_{t=0}$ (in particular, they hold for the $w$-normalised conormal derivative $\pd_{\nu_{w^{-1}A}}u$). Moroever,   one has the
representation
\begin{equation}
\label{eq:representationRegNeu}
\nabla_{w^{-1}A}u = e^{- t  DB_{0}}h^+ + S_{A}(\nabla_{w^{-1}A}u).
\end{equation}
for a unique $h^+\in \mH^+_{DB_{0}}$ and
\begin{equation}
\label{eq:representationRegNeu1}
\nabla_{w^{-1}A}u|_{t=0} = h^+ + h^-, \quad h^-=-\int_0^\infty  e^{sDB_{0}}E_0^-D \E_s (\nabla_{w^{-1}A}u)_s ds.
\end{equation}
Finally, there exists $u_{0}\in \dot H^1(w)$ (as defined in Lemma \ref{lem:gradient}) such that  $\nabla_{x}u_{0}= (g_{0})_{\ta}$ and one has $\|u_{t}-u_{0}\|\lesssim t$   and for $dw$ almost every $x_{0}\in \R^n$
\begin{equation}
\label{eq:CVaewsolreg}
\lim_{t\to 0}\ \barint_{\hspace{-6pt}W(t,x_{0})} u  \, d\uw=u_{0}(x_{0}).
\end{equation}
Remark \ref{rem} about replacing $d\uw$ by Lebesgue measure in the almost everywhere limit applies here too.

\end{thm}

\begin{thm}   \label{apriori_HSDir}
   Consider coefficients $w^{-1}A\in L^\infty(\R^{1+n}_+; \mL(\C^{m(1+n)}))$ such that $w^{-1}A$ is  accretive on $\mH^{0}$ and assume there exists $t$-independent measurable coefficients $A_0$
 such that $\| w^{-1}(A-A_0) \|_* <\infty$, or equivalently that $ \|\E\|_*\sim \| w^{-1}(A-A_{0})\|_{*}<\infty$ where $\E=B_{0}-B$ and $B=\widehat {w^{-1}A}, B_{0}=\widehat{w^{-1}A_{0}}$.

  Let  $u$     be a weak solution of
   $\divv A\nabla u=0$ in $\reu$ and assume that
  $\int_0^\infty \|\nabla_{t,x}u\|^2\, tdt<\infty$.
 Then $u= \hat u+c$ almost everywhere,
 for a unique constant $c\in \C^m$ and  $\hat u\in C([0,\infty); L^2(\R^n,w;\C^m))$ given by
 $\hat u= -v_{\perp}$ with
 \begin{equation}
\label{eq:representationDir}
v= e^{- t B_{0}D } \tilde h^+ + \tS_{A}(\nabla_{w^{-1}A}u),
\end{equation}
for a unique $ \tilde h^+\in \tE_{0}^+\mH$. Moreover,
\begin{equation}
\label{eq:representationDir1}
v_{0}=  \tilde h^+ + \tilde h^-
\ \mathrm{with}\  \tilde h^-= -\int_0^\infty  e^{sB_{0}D}\tE_0^- \E_s (\nabla_{w^{-1}A}u)_{s}\,  ds,
\end{equation}
and we call $-v$ the conjugate system associated to $u$.  In addition, we have $Dv=\nabla_{w^{-1}A}u$.

 Identifying the functions $u$ and $\hat u+c$, we have
  limits
$$
  \lim_{t\to 0} \| u_t -\hat u_0-c \| =0=
   \lim_{t\to \infty} \| u_t -c\|,
$$
for  $\hat u_0= -(\tilde h^+)_{\perp}\in L^2(\R^n, w;\C^m)$, and we have the estimates
$$
\|\hat u_{0}\|\lesssim \max(\|\tN(\hat u)\|, \sup_{t>0}\|\hat u_t \|)\lesssim \bigg(\int_0^\infty \|\nabla_{t,x}u\|^2\,  tdt\bigg)^{1/2}.
$$
Finally, for $dw$ almost every $x_{0}\in \R^n$,
\begin{equation}
\label{eq:CVaewsoldir}
\lim_{t\to 0}\ \barint_{\hspace{-6pt}W(t,x_{0})} u  \, d\uw=u_{0}(x_{0}).
\end{equation}
 Remark \ref{rem} about replacing $d\uw$ by Lebesgue measure in the almost everywhere limit applies here too.

\end{thm}

The representation formula suggests a possible construction of solution given $\tilde h^+$ provided $\|\E\|_{*}$ is sufficiently small. This is what leads to well-posedness results.

\begin{rem}
We also have the representation \eqref{eq:representationRegNeu} with both $e^{-t|DB_{0}|}$ and $h^+$ interpreted in a suitable  sense with Sobolev spaces of order $s=-1$. This point of view is developed more systematically in \cite{R} when $w=1$ and with Sobolev regularity $-1\le s \le0$. We refer the reader there to make the straightforward adaptation, as it is again an abstract argument.   We just warn the reader that the $\E$ in \cite{R} is not exactly the same as ours because the author assumed pointwise accretivity to simplify matters. One should use  representation \eqref{eq:firstformalSAdefn}  for $S_{A}$ instead.
\end{rem}

\begin{cor}\label{cor:aentcv}
Assume that $A$ satisfies $\| w^{-1}(A-A_0) \|_* <\infty$ for some $t$-independent $A_{0}$ and is such that all weak solutions $u$ to the system $\divv A \nabla u=0$ in a ball $B\subseteq \reu$ satisfy
 the local boundedness property
$$
\sup_{\alpha B}|u| \le C \left( \,\, \barint_{\hspace{-6pt}\beta B} |u|^2 d\uw\right)^{1/2},
$$
for any fixed constant $\alpha<\beta<1$,  with  $C$ independent of $u$ and  $B$.  Then any
weak solution $u$ with $\int_0^\infty \|\nabla_{t,x}u\|^2\, tdt<\infty$ or $\|\tN(\nabla_{t,x}u)\|<\infty$ converges non-tangentially almost everywhere to its boundary trace.
\end{cor}

The proof is a straightforward consequence of the more precise almost everywhere convergences
we stated in the previous section. We skip the details.
This result applies in particular to real equations as a consequence of \cite{FKS}.

\section{Well-posedness}\label{sec:solvability}

\subsection{Formulation and general results}

\begin{defn} Fix $w\in A_{2}(\R^n)$. Consider degenerate coefficients $A$ with $w^{-1}A\in L^\infty(\R^{1+n}_+; \mL(\C^{m(1+n)}))$ such that $w^{-1}A$ is  accretive on $\mH^{0}$.
 \begin{itemize}
  \item By the Dirichlet problem with coefficients $A$  being well-posed, we mean that
  given $\bphi\in L^2(\R^n,w;\C^m)$, there is a unique weak solution $u$
   solving (\ref{eq:divform}),  with $ \int_0^\infty \|\nabla_{t,x} u\|^2\,  tdt<\infty$
 and trace $u_0= \bphi$.
  \item By the regularity problem with coefficients $A$  being well-posed, we mean that
  given $\bphi\in L^2(\R^n,w; \C^{mn})$, where $\bphi$ satisfies $\curl_x \bphi=0$, there is a  weak solution $u$,
  unique modulo constants,
  solving  (\ref{eq:divform}),    with $\|\tN(\nabla_{t,x} u)\|<\infty$
  and  such that $\nabla_{x}u|_{t=0}= \bphi$.
   \item By the Neumann problem with coefficients $A$  being well-posed, we mean that
  given $\bphi\in L^2(\R^n,w;\C^m)$, there is a weak solution $u$,
  unique modulo constants,
  solving  (\ref{eq:divform}),  with $\|\tN(\nabla_{t,x} u)\|<\infty$
  and  such that $\pd_{\nu_{w^{-1}A}}u|_{t=0}= \bphi$.
\end{itemize}
We  write $A \in$ WP(BVP), if the corresponding boundary value problem (BVP) is well-posed with coefficients  $A$.
\end{defn}

We remark that the definition does not include  almost everywhere requirements. For the regularity and Neumann problems, one can make sense of the trace in a weak sense, but for the Dirichlet problem, the trace may not even make sense. However, as soon as we assume $\|w^{-1}(A-A_{0})\|_{*}<\infty$, which will be the case here, we know exactly the meaning of the trace from the results in Section \ref{sec:apriori}.

The most important observation following the \textit{a priori}  estimates in Theorems \ref{apriori_HSNeumann} and \ref{apriori_HSDir} is the fact that in the $t$-independent coefficient case, we completely identify the trace spaces: $ \mH^+_{DB}$ is the trace space of $w$-normalized conormal gradients for solutions with $\|\tN(\nabla_{t,x} u)\|<\infty$;
$ \mH^+_{BD}$ is the trace space of  conjugate systems $v$  for solutions with $ \int_0^\infty \|\nabla_{t,x} u\|^2\,  tdt<\infty$. In each case this is an isomorphism.

This leads to the following characterisation of well-posedness.

\begin{thm} Consider coefficients $w^{-1}A\in L^\infty(\R^{1+n}_+; \mL(\C^{m(1+n)}))$ such that $w^{-1}A$ is  accretive on $\mH_{0}$. Assume that $A$ has $t$-independent coefficients.  Let $B=\widehat {w^{-1}A}$.
Then $A\in$ WP(Reg)/WP(Neu)/WP(Dir)   if and only if
\begin{align*}
  \mH^+_{DB} \longrightarrow \sett{g\in L^2(\R^n,w;\C^{mn})}{\curl_x g=0} &: f\longmapsto f_\ta, \\
   \mH^+_{DB} \longrightarrow L^2(\R^n,w;\C^m) &: f\longmapsto f_{\perp}, \\
   \mH^+_{BD} \longrightarrow L^2(\R^n,w;\C^m) &: f\longmapsto f_{\perp},
\end{align*}
are isomorphisms respectively.
\end{thm}

Observe the change of space in the third line.

Let us mention a connection to so-called Rellich estimates. The isomorphisms imply the Rellich estimates
\begin{align*}
  \|f_{\perp}\|\lesssim \|f_{\ta}\|, \quad \forall f \in \mH^+_{DB}, \\\
  \|f_{\ta}\|\lesssim \|f_{\perp}\|, \quad \forall f \in \mH^+_{DB}, \\
  \|f_{\ta}\|\lesssim \|f_{\perp}\|, \quad \forall f \in  \mH^+_{BD},
\end{align*}
respectively. Assuming the Rellich estimates is not enough to conclude well-posedness because this only gives injectivity with closed range. The surjectivity  usually follows from  a continuity argument starting with a situation where one knows surjectivity. Thus, if, in a connected component of validity of a Rellich estimate, there is one $B$   for which surjectivity holds, then surjectivity holds for all $B$ in  this connected component and the corresponding BVP is well-posed for all $B$ in this connected component. Usually one considers $B=I$ so that, here,  $A=wI$. See \cite{AAM} for a discussion which applies \textit{in extenso}.  We remark that this depends on the continuous dependence on $B\in L^\infty$ of the projections $E_{0}^+$ and $\tE_{0}^+$, which follows from Theorem \ref{th:main}.
Let us mention also the duality principle between Dirichlet and Regularity, whose proof is the same as that of Proposition 17.6 in \cite{AA2}.

\begin{thm}\label{th:dualityDirReg} Let $w^{-1}A\in L^\infty(\R^{1+n}_+; \mL(\C^{m(1+n)}))$ such that $w^{-1}A$ is  accretive on $\mH^{0}$. Assume  there exists $t$-independent measurable coefficients $A_0$
 such that $\| w^{-1}(A-A_0) \|_* <\varepsilon$. If $\varepsilon$ is small enough, then $A\in$ WP(Dir) if and only if $A^*\in$ WP(Reg).
\end{thm}

We now turn to perturbation results for both $t$-dependent and $t$-independent coefficients. Adapting \cite{AAH, AAM}, see especially Lemma 4.3 in \cite{AAM}, one obtains that
each WP(BVP) is open under perturbation of $t$-independent coefficients in $wL^\infty$.
We refer to \cite{AA1}  for the proofs of the $t$-dependent perturbations, which carry over without change to our setting. We gather these observations together in the following statement.

\begin{thm}   \label{thm:Nellie} Assume the Neumann problem with $t$-independent $A_0$ is well-posed.
 Then there exist $\varepsilon_{0}>0$ and $ \varepsilon_{1}>0$ such that if $\| w^{-1}(A-A_1) \|_* <\varepsilon_{1}$ and $A_{1}$ has $t$-independent coefficients with $\| w^{-1}(A_{1}-A_0) \|_\infty <\varepsilon_{0}$, then the Neumann problem  with coefficients $A$ is well-posed.

The corresponding result holds when the Neumann problem is replaced by the regularity problem.

Moreover, for all such $A$, the solutions  $u$ of the BVP satisfy $$
   \|\tN(\nabla_{t,x} u)\| \approx \|g_0\| \approx \|\bphi\|,
 $$
 with $\bphi$ the $w$-normalized Neumann data or  the regularity data, and one has the limits and regularity estimates as described in Theorem \ref{apriori_HSNeumann}.

\end{thm}

With the duality principle above, we obtain.

\begin{thm}\label{thm: DirLip}
  Assume the Dirichlet problem with $t$-independent $A_0$ is well-posed.
 Then there exist $\varepsilon_{0}>0$ and $\varepsilon_{1}>0$ such that if $\| w^{-1}(A-A_1) \|_* <\varepsilon_{1}$ and $A_{1}$ has $t$-independent coefficients with $\| w^{-1}(A_{1}-A_0) \|_\infty <\varepsilon_{0}$, then the Dirichlet
  problem  with coefficients $A$ is well-posed.
Moreover, one has $$
   \|\tN(u)\| \approx \sup_{t>0}\|u_t \| \approx
   \bigg(\int_0^\infty \|\nabla_{t,x} u\|^2\,  tdt \bigg)^{1/2}\approx \|\bphi\|,
 $$
  if $\bphi$ is the Dirichlet data,
and one has the limits and regularity estimates described in Theorem \ref{apriori_HSDir}.
\end{thm}

\subsection{Well-posedness for $t$-independent  hermitian coefficients}

\begin{prop} Assume that  $A=A^*$ and that $A$ is $t$-independent and satisfies the usual degenerate boundedness and accretivity on $\mH^{0}=\clos{\ran(D)}$ conditions. Then the regularity, Neumann and Dirichlet problems with coefficients $A$ are well-posed.
 \end{prop}

\begin{proof} Let $B=\widehat {w^{-1}A}$ and
  $f\in E^+_0\mH=\mH^+_{DB}$.
Theorem \ref{apriori_HSNeumann} in the case of  $t$-independent coefficients implies that the vector field $F_t=e^{-tDB}f$ in $\R^{1+n}_+$ is such that $\pd_t F_t=-DB F_t$,
$\lim_{t\rightarrow\infty}F_t=0$ and $\lim_{t\rightarrow 0}F_t=f$.
Let $N:= \begin{bmatrix} -I & 0 \\ 0 & I\end{bmatrix}$
and note that $D N +N D=0$.
Now,    the definition of $B=\widehat {w^{-1}A}$ and   the Hermitian
condition  $A^*=A$ imply $B^*N=NB$.  Using the hermitian inner product  $(\ , \ )$  for $dw$, we have
\begin{multline*}
    \pd_t (N F_t,  BF_t)
  = (NDB F_t,  BF_t)+ (N F_t, BDB F_t) \\
  =  (NDB F_t, B F_t)+ (DB^*N F_t, B F_t)
  = ((N D+ D N)B F_t, B F_t) =0.
\end{multline*}
Hence, integrating in $t$ and taking into account the  limit at $\infty$ gives us $(Nf,Bf)=0$.  Thus, separating scalar and tangential parts, we obtain the Rellich equality:
 \begin{equation}   \label{eq:rellich}
(f,Bf)= 2(f_{\perp}, (Bf)_{\perp})=2 (f_{\ta}, (Bf)_{\ta}).
\end{equation}
Consider first the Neumann problem. It follows  from  (\ref{eq:rellich}) and the accretivity of $B$ on $\mH^{0}$ that
$$
 \kappa \|f\|^2 \le \re(f,Bf)= 2\re(f_{\perp}, (Bf)_{\perp})\lesssim 2\|B\|_{\infty} \|f_{\perp}\| \|f\|.
$$
This shows that $\|f\|\lesssim \|f_{\perp}\|$  for the Neumann map for any hermitian $A$, which implies that this map is injective with closed range. The continuity argument explained above implies that $A\in $ WP(Neu) provided that  $I\in$  WP(Neu). That
$I\in$ WP(Neu) can be seen from the equality $\|\nabla (-\Delta_{w})^{-1/2}u\|=\|u\|$ and
$$
\sgn(D)= \begin{bmatrix} 0 & (-\Delta_{w})^{-1/2}\divv_{w} \\ -\nabla (-\Delta_{w})^{-1/2}  & 0\end{bmatrix}.
$$
Thus, for $f\in \mH^{0}$, $f\in \mH^+_{D}$ if and only if $f_{\ta}= -\nabla (-\Delta_{w})^{-1/2}f_{\perp}
$, which in turn holds  if and only if $f_{\perp}= (-\Delta_{w})^{-1/2}\divv_{w}f_{\ta}$. This implies that the map used for solving the Neumann problem are invertible.

That $A\in$ WP(Reg) is proved in the similar way. Then, by Theorem \ref{th:dualityDirReg}, it follows that $A=A^* \in$ WP(Dir).
\end{proof}

\subsection{Well-posedness with algebraic structure and $t$-independent coefficients}

 Recall that we write our coefficients $A$ as a $2 \times 2$ block matrix. We say that it is {\em block lower-triangular} if the upper off-diagonal block $ A_{\perp\ta}$ is 0,  and {\em block upper-triangular} if the lower block  $A_{\ta\perp}$ is 0.

\begin{thm}\label{thm:triangular} We assume that $A$ is $t$-independent and satisfies the usual degenerate boundedness and accretivity conditions.
\begin{itemize}
  \item The Neumann problem  with block lower-triangular   coefficients $A$ is well-posed.
    \item The regularity problem  with block upper-triangular   coefficients $A$ is well-posed. More generally, it suffices for the off-diagonal lower block of $A$ to be divergence free and have real entries.
  \item  The Dirichlet problem  with block lower-triangular   coefficients $A$ is well-posed. More generally, it suffices for the off-diagonal upper block to be divergence free and have real entries.\end{itemize}

\end{thm}

Let us clarify  the statements above. The off-diagonal lower block is $ A_{\ta\no}= (A^{\alpha,\beta}_{i,0})_{i=1,\ldots, n}^{\alpha,\beta= 1,\ldots, m}$. Real entries means that all these coefficients are real: it guarantees that $A$ and $$A'=A-\begin{bmatrix} 0 & -A_{\ta\no}^t \\ A_{\ta\no}  & 0\end{bmatrix}=\begin{bmatrix} A_{\no\no} & A_{\no\ta}+A_{\ta\no}^t \\ 0  & A_{\ta\ta}\end{bmatrix} $$
have the same accretivity  bounds. The divergence free condition is $\sum_{i=1}^n\pd_{i}A^{\alpha,\beta}_{i,0}=0$
for all $\alpha,\beta$. It implies that weak solutions with coefficients $A$ or $A'$ are the same as can be seen by integrating by parts. In other words, we can reduce matters the special case where the  off-diagonal lower block is zero
This possibility does not appear to be available for the Neumann problem because the conormal derivative depends on the coefficients.

  The proof of this theorem is obtained by a line by line adaptation of  \cite{AMM} to the weighted setting using well-posedness  of the three problems (modulo constants) in the class  of energy solutions, that is,  having finite energy $\int_{\reu} |\nabla u|^2 d\uw<\infty$.  We  mention that to carry out the algebra  there, one should replace the standard Riesz transforms by the Riesz transforms $\mR_{w}$ defined in Lemma \ref{lem:gradient}. We leave details to the interested reader.

\section{Non-tangential maximal estimates and Fatou type results}\label{sec:NTmax}

Recall that $\uw$ is the $A_{2}$ weight on $\R^{1+n}$ defined by $\uw(t,x)=w(x)$ and that   $\uw$ and $w$ have identical $A_{2}$ constants.
Writing equations $\divv A \nabla u =0$  as $\divv_{\uw} (w^{-1}A\nabla u)= 0$ allows one to carry  some proofs to the degenerate case  without much change from the non-degenerate case. We quote two results we will be using. The first is the usual Caccioppoli inequality with a completely analogous proof:
all weak solutions $u$ in a ball  $ B=B(\bx,r) \subseteq \reu$ of  $\divv A \nabla u=0$ enjoy the Caccioppoli inequality \begin{equation}
\label{eq:caccio}
{}\int_{\alpha B} |\nabla u|^2 \, d\uw \le Cr^{-2} {}\int_{\beta B}  |u|^2 \, d\uw
\end{equation}
for any $0<\alpha<\beta<1$,
 the implicit constants depending only on $\|w^{-1}A\|_{\infty}$ the accretivity constant of $w^{-1}A$, $n$, $m$, $\alpha$ and $\beta$.

The second one is a corollary of this and Poincar\'e inequalities: there exists $1<p_{1}<2$ such that for any $p_{1}<p<2$, all weak solutions $u$ in a ball $B=B(\bx,r) \subseteq \reu$ of  $\divv A \nabla u=0$ enjoy the reverse H\"older inequality  \begin{equation}
\label{eq:meyers}
\bigg(\barint_{\hspace{-6pt}\alpha B} |\nabla u|^2 \, d\uw\bigg)^{1/2} \lesssim \bigg(\barint_{\hspace{-6pt}\beta B}  |\nabla u|^p \, d\uw\bigg)^{1/p}
\end{equation}
for any $0<\alpha<\beta<1$,
 the implicit constants depending only on $\|w^{-1}A\|_{\infty}$ the accretivity constant of $w^{-1}A$, $n$, $m$, $p$, $\alpha,\beta$ and the $A_{2}$ constant of $w$. The usual proof follows from Caccioppoli's inequality applied to $u-u_{\beta B}$ and using the Poincar\'e inequality with  the gradient
in $L^p$. Here, to be able to do that we need $\uw \in A_{2-\varepsilon}$ for $\varepsilon>0$ but this can be done  for some $\varepsilon>0$ depending only in the size  of the $A_{2}$ constant of $\uw$. Then, one can use \cite[Theorem 1.5]{FKS} which asserts that the gain of exponent in the Poincar\'e inequality for an $A_{p}$ weight on $\R^{1+n}$ is at least $p \frac{1+n}n$. So for $p> \frac {2n}{1+n}$  (and $p>2-\varepsilon$) we are done.

We continue with the analogue of Lemma 10.3 in \cite{AA1}.

\begin{lem}  \label{lem:lqoffdiag}
  Let $B$ be $t$-independent, bounded on $\mH$ and accretive on $\mH^{0}= \clos{\ran(D)}$. Let $T=DB$ or $BD$. Then there exists $1<p_{0}<2$, such that for $p_{0}<q<p_{0}'$,
$$
  \| (I+it T)^{-1} f \|_{q}\le  {C_N}\bigg(1+\frac{\dist(E,F)}{t}\bigg)^{-N} \|f\|_{q}
$$
for all   integer $N$,   $t>0$ and sets $E,F\subseteq \R^n$ such that $\supp f\subseteq F$, with $C_{N}$ independent of $f, t, E, F$.  Here $\dist(E,F):= \inf\sett{|x-y|}{x\in E, y\in F}$ and $\|\ \|_{q}$ are the weighted $L^q$ norms.
\end{lem}

\begin{proof} It suffices to prove the lemma for $T=DB$ as then it holds for $T^*=B^*D$, and hence for $BD$ upon changing $B^*$ to $B$.

  For $q=2$, this is contained in  Lemma \ref{lem:odd}.  By interpolation, it suffices to estimate the operator norm of $ (I+it DB)^{-1}$ on $L^q(\R^n,w;\C^{m(1+n)})$, uniformly for $t$.

  To this end, assume that $(I+ it DB)\tilde f= f$.
  As in Proposition~\ref{prop:divformasODE}, but replacing $\pd_t$ by $(it)^{-1}$, this equation is equivalent to
$$
\begin{cases}
  (w^{-1} A\tilde g)_\no + it\divv_w(w^{-1} A\tilde g)_\ta = (w^{-1} A g)_\no, \\
  \tilde g_\ta - it \nabla_x\tilde g_\no = g_\ta,
\end{cases}
$$
where $A, g$ and $ \tilde g$ are related to $B, f $ and $ \tilde f$ respectively, as in Proposition~\ref{prop:divformasODE}.
Using the second equation to eliminate $\tilde g_\ta$ in the first, shows that $\tilde g_\no$ satisfies the divergence form
equation
$$
  L_{t}\tilde g_\no:= \begin{bmatrix} 1 & it\divv_w \end{bmatrix}
   (w^{-1}A)
    \begin{bmatrix} 1 \\ it\nabla_x \end{bmatrix}
    \tilde g_\no =
     \begin{bmatrix} 1 & it\divv_w \end{bmatrix}
      \begin{bmatrix} w^{-1}A_{\no\no} g_\no \\ - w^{-1}A_{\ta\ta} g_\ta \end{bmatrix}.
$$
Let $r(w)<2$ be the infimum of those exponents $q$ for which $w\in A_{q}$. For $r(w)<q<r(w)
'$, $L_{t}$ is bounded from $W^{1,q}(\R^n,w; \C^{m})$ equipped with the norm
$\|u\|_{q}+ t\|\nabla u \|_{q}$ into $W^{-1, q}(\R^n,w;\C^m)=(W^{1, q'}(\R^n, w;\C^m))'$ for the duality $\langle \ , \ \rangle_{w}$, uniformly in $t$. (Observe that the $A_{2}$ constant is unchanged by scaling, so we can change variable to reduce to $t=1$ up to changing $w(x)$ to $w(tx)$.) The accretivity condition on $\mH^{0}$ tells us  that
$$
\re \langle   L_{t}u,u \rangle_{w} \ge \kappa (\|u\|_{2}^2 + t^2\|\nabla u \|_{2}^2).$$
 Hence $L_{t}$ is an isomorphism for $p=2$. Using Proposition \ref{prop:interpolation} and
the stability result of {\v{S}}ne{\u\i}berg~\cite{sneiberg}, it follows that
 $L_{t}$ is an isomorphism for $p_0<q<p_{0}'$ for some $r(w)<p_{0}<2$. This gives us the desired estimate in the weighted $L^q(w)$ norms,
$$
  \|\tilde f\|_q\approx \|\tilde g\|_q \lesssim \|\tilde g_\no\|_q+ t\|\nabla_x \tilde g_\no\|_q + \|g_\ta\|_q
  \lesssim \|g\|_q\approx \|f\|_q.
$$
\end{proof}

We state a simple but useful lemma.

\begin{lem}\label{lem:SFEimplies0limit} Let $f$ be a function
 satisfying the square function estimate $\int_{0}^\infty \|f_{t}\|^2\, \frac{dt}t <\infty$. Then the Whitney averages of $f$ converge $dw$ almost everywhere in $L^2$ sense to 0.
\end{lem}

\begin{proof} Recall that $c_{0},c_{1}$ are the parameters for the Whitney box $W(t,x)$. Let
$$
h(t,x):= \int_{|x-y|\le c_{1}c_{0}s, \, s<c_{0}t}  |f(s,y)|^2\frac{d\uw(s,y)}{s w(B(y,c_{1}s))}.
$$
It is an increasing function of $t$ with
$$
\int_{\R^n} h(t,x) \, dw(x) \lesssim
\int_{0}^{c_{0}t} \|f_{s}\|^2\, \frac{ds}s \to 0
$$
as $t\to 0$.  Thus $h(t,x)$ converges to 0 as $t\to 0$ for $dw$ almost every $x$. We conclude observing that
$W_{2}f(t,x)^2 \lesssim h(t,x)$.
\end{proof}

We now turn to the proof of Theorem \ref{thm:NTmaxandaeCV}.

\begin{proof}[Proof of \eqref{eq:Ntmax}.]  Here it remains to show $\|\tN(e^{-t|T|}h) \| \sim \|h\|$ whenever $h\in \clos{\ran(T)}$ and $T=DB$ or $BD$ with $B$ being $t$-independent (we drop the  subscript 0).  The bound from below is easy:
$$
\|h\|^2 = \lim_{t\to 0}  \frac 1t\int_t^{2t} \| e^{-s|T|}h \|^2\,  ds \lesssim \| \tN(e^{-s|T|}h) \|^2
$$
from Lemma \ref{lem:XlocL2}.

Let us see the converse when $T=DB$.  Split $h\in \clos{\ran(DB)}=\mH^{0}$ as $h=h^++h^-$ according to $h^\pm= \chi^\pm(DB)h$. Then $e^{-t|DB|}h^\pm$ is the $w$-normalized conormal gradient of a weak solution $u^\pm$ in the upper-(resp. lower-)half space. Thus one can use
\eqref{eq:meyers} and estimate $\tN(e^{-t|DB|}h^\pm)$ by $\tN^p(e^{-t|DB|}h^\pm)$ for some $p<2$ which we can take larger than $p_{0}$ of the Lemma \ref{lem:lqoffdiag}. From here follow the proof of Lemma 2.52 in \cite{AAH}: decompose $  e^{-t|DB|}=\psi(tDB) + (I+itDB)^{-1}$ and use  the quadratic estimate \eqref{eq:psiT} and Lemma \ref{lem:XlocL2} to obtain $$
\| \tN^p(\psi(tDB)h^\pm)\|^2 \lesssim \int_{0}^\infty \|\psi(tDB)h^\pm\|^2 \, \frac{dt}{t} \lesssim \|h^\pm\|^2
$$
and   Lemma \ref{lem:lqoffdiag} to obtain the pointwise bound
$$
\tN^p((I+itDB)^{-1}h^\pm) \le C M_{w}(|h^\pm|^p)^{1/p}, $$
where $M_{w}$ is the Hardy-Littlewood maximal operator with respect to $dw$.

The result for $T=BD$ follows from that of $DB$:  If $g\in \clos{\ran(BD)}$, then $B^{-1}g=h\in \clos{\ran(DB)}$  with $\|h\| \sim \|g\|$ and  $e^{-t|BD|}g=B e^{-t|DB|}h$. Thus
$$
 \| \tN(e^{-t|BD|}g) \| =  \| \tN(B e^{-t|DB|}h) \| \leq \|B\|_{\infty}  \| \tN( e^{-t|DB|}h) \|  \sim \|h\|.
$$
\end{proof}

\begin{proof}[Proof of \eqref{eq:CVae}.] This time we begin with $T=BD$.  We can split $h=h_{N}+ h_{R}$ where $h_{N}\in \nul(BD) $ and $h_{R}\in \clos{\ran(BD)}. $ As $e^{-s|BD|}h_{N}=h_{N}$,  the almost everywhere limit
$\lim_{t\to 0}\ \barint_{\hspace{-2pt}W(t,x_{0})} |e^{-s|BD|}h_{N}-h_{N}(x_{0})|^2\, d\uw=0$ follows from the Lebesgue differentiation theorem.  We can thus assume that $h\in \clos{\ran(BD)}$.

 Pick  a Lebesgue point  $x_{0}$ for the conditions
\begin{equation}
\label{eq:leb2}
\lim_{t\to 0}\ \barint_{\hspace{-6pt}B(x_{0},t)} |h-h(x_{0})|^2\, dw=0= \lim_{t\to 0}\ \barint_{\hspace{-6pt}B(x_{0},t)} |w^{-1}-w^{-1}(x_{0})|^2\, dw.
\end{equation}
The second equality is possible  since $w^{-1} \in L^2_{\loc}(w)$ as $w\in A_{2}$.
As a consequence of \eqref{eq:leb2} and \eqref{eq:dxdwAp} with $p=2$, we have
\begin{equation}
\label{eq:leb1}
\lim_{t\to 0}\ \barint_{\hspace{-6pt}B(x_{0},t)} w \, dx=w(x_{0}), \qquad  \lim_{t\to 0}\ \barint_{\hspace{-6pt}B(x_{0},t)} w^{-1} \, dx = w^{-1}(x_{0})
\end{equation}
and
$$\lim_{t\to 0}\ \barint_{\hspace{-6pt}B(x_{0},t)} |h-{(hw)(x_{0})}w^{-1}|^2\, dw=0.
$$
Write as above, $e^{-s|BD|}h=\psi(sBD)h + (I+isBD)^{-1}h$. The quadratic estimate  \eqref{eq:psiT}  and Lemma \ref{lem:SFEimplies0limit} imply that
$$
\lim_{t\to 0}\ \barint_{\hspace{-6pt}W(t,x_{0})} | \psi(sBD)h|^2 \, d\uw =0
$$
for $dw$ almost every $x_{0}\in \R^n$. Now the key point is that $D\begin{bmatrix}
      c_{\perp}    \\
      c_{\ta}w^{-1}
\end{bmatrix}= 0$ if $c$ is a constant, thus $(I+isBD)^{-1} \begin{bmatrix}
      h_{\perp}(x_{0})    \\
      (h_{\ta}w)(x_{0})w^{-1}
\end{bmatrix}= \begin{bmatrix}
      h_{\perp}(x_{0})    \\
      (h_{\ta}w)(x_{0})w^{-1}
\end{bmatrix}.$  It follows that
\begin{align*}
& (I+isBD)^{-1}h -h(x_{0}) \\
&
=(I+isBD)^{-1}h- \begin{bmatrix}
      h_{\perp}(x_{0})    \\
      (h_{\ta}w)(x_{0})w^{-1}
\end{bmatrix} + \begin{bmatrix}
      0    \\
      (h_{\ta}w)(x_{0})w^{-1} -h_{\ta}(x_{0})
\end{bmatrix}\\
&= (I+isBD)^{-1} \begin{bmatrix}
      h_{\perp}- h_{\perp}(x_{0})      \\
     h_{\ta}-  (h_{\ta}w)(x_{0})w^{-1}
\end{bmatrix}  + (h_{\ta}w)(x_{0})\begin{bmatrix}
      0    \\
      w^{-1} -w^{-1}(x_{0})
\end{bmatrix}.
\end{align*}
By the assumption on  $x_{0}$,
$$
\lim_{t\to 0}\ \barint_{\hspace{-6pt}W(t,x_{0})} \bigg| (h_{\ta}w)(x_{0})\begin{bmatrix}
      0    \\
      w^{-1} -w^{-1}(x_{0})
\end{bmatrix} \bigg|^2\, d\uw=0.
$$
Decomposing the inner function of the other term using annuli centred around $B(x_{0},t)$ and using Lemma \ref{lem:odd}, we have that
$$
 \barint_{\hspace{-6pt}W(t,x_{0})} \bigg | (I+isBD)^{-1} \begin{bmatrix}
      h_{\perp}- h_{\perp}(x_{0})    \\
     h_{\ta}-  (h_{\ta}w)(x_{0})w^{-1}
\end{bmatrix} \bigg|^2 \, d\uw
$$
is bounded by
\begin{equation}
\label{eq:NTbound}
 \sum_{j\ge 1} 2^{-jN}  (w(B(x_{0},t)))^{-1} \int_{B(x_{0}, 2^j t)} \bigg |  \begin{bmatrix}
      h_{\perp}- h_{\perp}(x_{0})    \\
     h_{\ta}-  (h_{\ta}w)(x_{0})w^{-1}
\end{bmatrix} \bigg|^2 \, dw.
\end{equation}
The scalar term is bounded by
\begin{align*}
&\sum_{j\ge 1} 2^{-j(N-d_{w})}   \barint_{\hspace{-6pt}B(x_{0}, 2^j t)}  |
      h_{\perp}- h_{\perp}(x_{0})     |^2 \, dw
      \\
      \lesssim &
      \sup_{\tau\le \sqrt t }\  \barint_{\hspace{-6pt}B(x_{0}, \tau)}  |        h_{\perp}- h_{\perp}(x_{0})
     |^2 \, dw + \sqrt t M_{w}(|h|^2)(x_{0})
 \end{align*}
as can be seen using the doubling condition on $w$ ($d_{w}$ being its homogeneous dimension) and
breaking the sum up at $j_{0}$ with $2^{j_{0}}\sim 1/ \sqrt t$ if $N\ge d_{w}+1$,
where $M_{w}$ is the Hardy-Littlewood maximal operator with respect to $dw$. The weak type (1,1) estimate of $M_{w}$ implies that  $M_{w}(|h|^2)(x_{0})<\infty$ for almost every $x_{0}\in \R^n$.  Hence,  the scalar term  goes to 0  as $t\to 0$ at those $x_{0}$ which meet all these requirements.
The tangential term  in \eqref{eq:NTbound} is bounded by
\begin{align*}
&\sum_{j\ge 1} 2^{-j(N-d_{w^{-1}})}   \barint_{\hspace{-6pt}B(x_{0}, 2^j t)}  |
      h_{\ta}- h_{\ta}(x_{0})     |^2 \, dw
      \\ + &  \sum_{j\ge 1} 2^{-jN} |h_{\ta}(x_{0})|^2   (w(B(x_{0},t)))^{-1} \int_{B(x_{0}, 2^j t)} |w^{-1}(x_{0})-w^{-1}|^2\,dw.
 \end{align*}
 The first sum can be treated as above. For the second when  $j\le j_{0}$, we also do as above. For $j\ge j_{0}$, we write $|w^{-1}(x_{0})-w^{-1}|^2 \le 2 w^{-2}(x_{0}) + 2w^{-2}$. The first term leads to a bound $\sqrt t  |h_{\ta}(x_{0})|^2 w^{-2}(x_{0})$ if $N\ge d_{w}+1$. For the second term,
observe that
 $$\int_{B(x_{0}, 2^j t)} w^{-2}dw= w^{-1}(B(x_{0}, 2^j t)) \lesssim 2^{jd_{w^{-1}}} w^{-1}(B(x_{0},  t))$$ and  that  by \eqref{eq:leb1}$$(w(B(x_{0},t)))^{-1}w^{-1}(B(x_{0},  t)) \to 1/w^2(x_{0}), \quad t\to 0.$$
 Thus, if also $N\ge d_{w^{-1}}+1$, then the tangential term in \eqref{eq:NTbound} tends to 0.

We now turn to the proof for $T=DB$.  If $g\in \nul(DB)$, \eqref{eq:CVae} is a consequence of the Lebesgue differentiation theorem for the measure $dw$ on $\R^n$ as $e^{-s|DB|}g=g$ for all $s$.
Now consider $g\in \clos{\ran(DB)}$. As
$$
\lim_{t\to 0}\ \barint_{\hspace{-6pt}W(t,x_{0})} |g-g(x_{0})|^2\, d\uw= \lim_{t\to 0}\ \barint_{\hspace{-6pt}B(x_{0},c_{1}t_{0})} |g-g(x_{0})|^2\, dw= 0$$
for almost every $x_{0}\in \R^n$,
 it is enough to show the almost everywhere limit
$$
\lim_{t\to 0}\ \barint_{\hspace{-6pt}W(t,x_{0})} |e^{-s|DB|}g-g|^2\, d\uw=0.
$$
Write
$e^{-s|DB|}g-g= \psi(sDB)g + (I+isDB)^{-1}g-g$. The quadratic estimate  \eqref{eq:psiT} and Lemma \ref{lem:SFEimplies0limit} imply that
$$
\lim_{t\to 0}\ \barint_{\hspace{-6pt}W(t,x_{0})} | \psi(sDB)g|^2 \, d\uw =0
$$
for almost every $x_{0}\in \R^n$. Now $(I+isDB)^{-1}g-g= -isDh_{s}$ with $ h_{s}=B(I+isDB)^{-1}g= (I+isBD)^{-1}(Bg)$ and $Bg\in \mH$. Let
 $$\tilde h_{s}:= (I+isBD)^{-1}(Bg)- \begin{bmatrix}
      (Bg)_{\perp}(x_{0})    \\
      ((Bg)_{\ta}w)(x_{0})w^{-1}
\end{bmatrix}.$$ As above we have that $Dh_{s}=D\tilde h_{s}$. We now apply the following local coercivity inequality on $\R^n$: for any $u\in \mH_{\loc}$ with $Du\in \mH_{loc}$ and any ball $B(x,r)$ in $\R^n$,
\begin{equation}
\label{eq:localcoerc}
\int_{B(x,r)} |Du|^2\, dw \lesssim \int_{B(x,2r)} |BDu|^2\, dw + r^{-2} \int_{B(x,2r)} | u|^2\, dw,
\end{equation}
where the implicit constant depends only on the ellipticity constants of $B$, the dimension and $m$. (Of course, if $B^{-1}$ is a multiplication operator, this inequality improves considerably.)

Applying this inequality  to $u=\tilde h_{s}$, using $Dh_{s}=D\tilde h_{s}$ and integrating with respect to $s$ implies
\begin{align*}
\label{}
\barint_{\hspace{-6pt}W(t,x_{0})} |isDh_{s}|^2\, d\uw    & \lesssim  \barint_{\hspace{-6pt}\widetilde W(t,x_{0})} |isBDh_{s}|^2\, d\uw  +  \barint_{\hspace{-6pt}\widetilde W(t,x_{0})} |\tilde h_{s}|^2\, d\uw  \\
    & \lesssim  \barint_{\hspace{-6pt}\widetilde W(t,x_{0})} |(I+isBD)^{-1}(Bg) -Bg|^2\, d\uw
    \\
  &  \qquad +  \barint_{\hspace{-6pt}\widetilde W(t,x_{0})} \bigg|Bg- \begin{bmatrix}
      (Bg)_{\perp}(x_{0})    \\
      ((Bg)_{\ta}w)(x_{0})w^{-1}
\end{bmatrix}\bigg|^2\, d\uw,
\end{align*}
where $\widetilde W(t,x_{0})$ is a slightly expanded version of $W(t,x_{0})$ and, in the last inequality, we have written
$$\tilde h_{s}= (I+isBD)^{-1}(Bg) -Bg + Bg- \begin{bmatrix}
      (Bg)_{\perp}(x_{0})    \\
      ((Bg)_{\ta}w)(x_{0})w^{-1}
\end{bmatrix}.
$$
The  last two integrals have been shown to converge to $0$ as $t\to 0$ for almost every $x_{0}\in \R^n$ in the argument for $BD$.  It only remains to prove   \eqref{eq:localcoerc}.

For this inequality, we let $\chi$ be a smooth scalar-valued function with $\chi=1$ on $B(x,r)$, $\chi$ supported in $B(x,2r)$  and  $|\nabla \chi|\lesssim r^{-1}$.  Using the commutator identity \eqref{eq:comm} between $\chi$ and $D$, we have
$$
\int_{B(x,r)} |Du|^2\,  dw  \leq \int  |\chi Du|^2\,  dw  \lesssim  \int  |D(\chi u)|^2\,  dw + \int  |\nabla\chi |^2| u|^2\,  dw.
$$
Since $B$ is accretive on $\ran(D)$ and  $\chi u \in \dom(D)$, we have $ \int  |D(\chi u)|^2\,  dw \lesssim  \int  |BD(\chi u)|^2\,  dw$. Now, we use again the commutation between $\chi$ and $D$ together with $\|B\|_{\infty}$. This proves \eqref{eq:localcoerc}.\end{proof}

\bibliographystyle{acm}

\end{document}